\documentclass[a4paper,11pt]{scrartcl}
\usepackage{amsmath,amssymb}
\usepackage{amsthm}
\usepackage{graphicx}
\usepackage[square,sort&compress,numbers]{natbib}
\usepackage{url}
\usepackage{type1cm}
\usepackage{mathtools}\mathtoolsset{showonlyrefs=false}
\usepackage{booktabs}

\usepackage{subfig}

\newcommand{\OMIT}[1]{\textbf{[OMIT:} #1 \ \textbf{ --- end OMIT] }}  
   \renewcommand{\OMIT}[1]{}            

\newcommand{\Authornote}{\renewcommand{\thefootnote}{\fnsymbol{footnote}}}
\newcommand{\authornote}{\Authornote\footnote}

\theoremstyle{plain}
\newtheorem{theorem}{Theorem}

\theoremstyle{definition}

\theoremstyle{remark}
\newtheorem{remark}{Remark}

\newcommand{\refthm}[1]{Theorem~\ref{#1}}

\newcommand{\reffig}[1]{Figure~\ref{#1}}
\newcommand{\reftab}[1]{Table~\ref{#1}}

\newcommand{\finbox}{\nolinebreak\hfill{\small $\blacksquare$}}

\newcommand{\MIN}{\mathop{\mathrm{Minimize}}}

\newcommand{\MAX}{\mathop{\mathrm{Maximize}}}

\newcommand{\ST}{\mathop{\mathrm{subject~to}}}

\newcommand{\domain}{\mathop{\mathrm{dom}}\nolimits}

\renewcommand{\Re}{\ensuremath{\mathbb{R}}}

\newcommand{\overRe}{\ensuremath{\mathbb{R}\cup\{+\infty\}}}

\newcommand{\bi}[1]{\ensuremath{\boldsymbol{#1}}}

\newcommand{\rr}[1]{\ensuremath{\mathrm{#1}}}

\newcommand{\KC}{\ensuremath{\mathcal{K}}}

\newcommand{\SC}{\ensuremath{\mathcal{S}}}

\begin{document}

\begin{center}
  {\Large\bfseries\sffamily%
  Exploiting Lagrange duality for topology optimization }\\
  \medskip
  {\Large\bfseries\sffamily%
  with frictionless unilateral contact }%
  \par%
  \bigskip%
  {
  Yoshihiro Kanno~\authornote[2]{%
    Mathematics and Informatics Center, 
    The University of Tokyo, 
    Hongo 7-3-1, Tokyo 113-8656, Japan.
    E-mail: \texttt{kanno@mist.i.u-tokyo.ac.jp}. 
    }
  }
\end{center}

\begin{abstract}
  This paper presents tractable reformulations of topology optimization 
  problems of structures subject to frictionless unilateral contact 
  conditions. 
  Specifically, we consider stiffness maximization problems of trusses 
  and continua. 
  Based on the Lagrange duality theory, we derive formulations that do 
  not involve complementarity constraints. 
  It is often that a structural optimization problem with contact 
  conditions is formulated as a mathematical programming problem with 
  complementarity constraints (MPCC problem). 
  However, MPCC usually requires special treatment for numerical 
  solution, because it does not satisfy standard constraint 
  qualifications. 
  In contrast, to the   formulation presented in this paper, we can 
  apply standard optimization approaches. 
  Numerical experiments of trusses and continua are performed to examine 
  efficiency of the proposed approach. 
\end{abstract}

\begin{quote}
  \textbf{Keywords}
  \par
  Topology optimization; 
  stiffness maximization; 
  nonsmooth mechanics; 
  unilateral contact; 
  Lagrange duality; 
  second-order cone programming. 
\end{quote}

\section{Introduction}

In this paper, we explore topology optimization of elastic 
structures (specifically, trusses and continua) subject to 
frictionless unilateral contact conditions. 
We assume that the structure can possibly make contact with the surface 
of a rigid obstacle. 
The obstacle is fixed in space, and the contact takes place without 
friction and adhesion. 
In this problem setting, we seek to find a structural design that has 
the maximal stiffness. 
This problem has quite a long history in the field of structural 
optimization; see a survey by \citet{HKP99survey} for early contributions. 
This paper attempts to shed new light on this problem from a perspective 
of Lagrange duality theory. 

An equilibrium state of a unilateral contact problem corresponds 
to a solution of a complementarity problem \citep{MR02,Wri06,Kan11}. 
Therefore, it is often (and also natural) that a structural optimization 
problem with unilateral contact is formulated as a 
{\em mathematical programming problem with complementarity constrains\/}
(MPCC problem); 
such formulations can be found in, e.g., \cite{Tin99,Hil00,SK10,HKP99}. 
Special treatment is usually required to solve an MPCC problem, because 
any feasible solution to an MPCC problem does not satisfy standard 
constraint qualifications \citep{LPR96}. 
A typical remedy is to apply reformulation and smoothing to the 
complementarity constraints so that the structural optimization problem 
can be handled with a conventional {\em nonlinear programming\/} (NLP) 
approach \citep{Tin99,Hil00} or a gradient-based topology optimization 
approach \citep{SK10}. 
Since the state variables (e.g., displacements, contact reactions, etc.) 
of a contact problem are in general nonsmooth with 
respect to design variables, another remedy is to apply a nonsmooth 
optimization method to the structural optimization problem. 
For example, \citet{Sta95} used a bundle method, and 
\citet{PP97} adopted a subgradient method. 
Alternatively, in spite of awareness of this nonsmoothness, 
a conventional gradient-based topology optimization approach (e.g., 
the method of moving asymptotes) were sometimes adopted \citep{KR95,HK12}. 
\citet{Str10,Str12} performed sensitivity analysis by using 
the (augmented) Lagrangian for the contact problem. 
As another approach for dealing with complementarity 
constraints in a structural optimization problem, 
\citet{HKP99} used a penalty interior-point method for MPCC. 

Other than approaches based on complementarity problems, the 
unilateral contact conditions in a structural optimization problem have 
been treated (often in approximated manners) by using 
so-called penalty methods 
(e.g., with gap-elements) \citep{LLK16,MA04,Fan06,KPC01,KYC02}. 
\citet{LM15} adopted the so-called stabilized Lagrangian multiplier 
method in the extended FEM \citep{GMM07}. 

The approach presented in this paper is different from the ones in 
literature cited above. 
Namely, by using the Lagrange duality, we recast the topology 
optimization problem under consideration as a standard optimization 
problem (in the sense that, unlike MPCC problems, it satisfies standard 
constraint qualifications). 
Moreover, this approach does not resort to any approximation, which 
means that an optimal solution of the proposed formulation precisely 
satisfies unilateral contact conditions. 

More specifically, we show that the stiffness maximization problem of 
trusses subject to unilateral contacts can be recast as a 
{\em second-order cone programming\/} (SOCP) problem. 
This is a convex optimization problem, and can be solved efficiently 
with a primal-dual interior-point method \citep{AL12}. 
We next extend the presented formulation to continua. 
Here, we adopt the conventional 
{\em solid isotropic material with penalization\/} (SIMP) 
approach \citep{BS99} with the density filter \citep{BT01,Bou01}. 
Due to the SIMP penalization, the formulation extended to continua 
is nonconvex. 
However, this formulation is suitable for application of a sequential 
SOCP approximation method, where the SIMP penalizations on densities are 
sequentially linearized. 
Thus, the proposed reformulation for continua can also be solved (in, in 
turn, a local sense) with a standard mathematical optimization approach 
(i.e., it does not require any special treatment for complementarity 
constraints). 

From the perspective of convexity in truss problems, the approach 
presented in \citet{KZZ98} is also of interest. 
Namely, they proposed to solve a convex optimization problem, which is 
obtained by eliminating the design variables (i.e., the member 
cross-sectional areas); similar approaches can be found also 
in \cite{BKNZ00,KPR95}. 
In section~\ref{sec:truss.relation}, we establish clear relationship 
between the approach presented in this paper and the one in \cite{KZZ98}, 
by using the Legendre--Fenchel transform and the minimax theorem. 
This analysis provides us with a deeper understanding of the source of 
convexity in these two different approaches. 
From a practical point of view, one significant advantage of the 
approach in this paper over the one in \cite{KZZ98} is that the former 
retains the member cross-sectional areas as optimization variables while 
the latter has eliminated them. 
Therefore, we can incorporate diverse constraints on truss design into 
the formulation in this paper. 
In the numerical experiments described in section~\ref{sec:ex.1}, we 
demonstrate some concrete examples of such additional design constraints.


The paper is organized as follows. 
In section~\ref{sec:truss}, we derive an SOCP formulation for trusses, 
by using the Lagrange duality theory. 
We also pursue investigation of relations between this formulation and 
existing other formulations \citep{KZZ98,KPR95,PP97}. 
Section~\ref{sec:body} extends the formulation presented in 
section~\ref{sec:truss} to continua. 
In section~\ref{sec:ex}, we perform numerical experiments on trusses and 
continua. 
In section~\ref{sec:conclusions}, we draw some conclusions.


In our notation, ${}^{\top}$ denotes the transpose of a vector or a matrix. 
All vectors are column vectors. 
For notational simplicity, we often write $(n+m)$-dimensional column 
vector $(\bi{x}^{\top},\bi{y}^{\top})^{\top}$ consisting of 
$\bi{x} \in \Re^{n}$ and $\bi{y} \in \Re^{m}$ as $(\bi{x},\bi{y})$. 
For a vector $\bi{x} \in \Re^{n}$, the notation $\|\bi{x}\|$ designates 
its Euclidean norm, i.e., $\|\bi{x}\| = \sqrt{\bi{x}^{\top}\bi{x}}$. 
We use $\bi{1}$ to denote an all-ones column vector. 
The notation $\SC^{n}$ designates the set of 
$n \times n$ real symmetric matrices. 
For a proper function $f : \Re^{n} \to \overRe$, its conjugate function 
is defined by 
\begin{align*}
  f^{*}(\bi{x}^{*})
  = \sup\{ \bi{x}^{\top} \bi{x}^{*} - f(\bi{x}) 
  \mid \bi{x} \in \Re^{n} \} , 
\end{align*}
where $f \mapsto f^{*}$ is called the Legendre--Fenchel transform 
(a.k.a.\ the Fenchel transform). 
We use $\Re_{+}^{n}$ to denote the nonnegative orthant, i.e., 
$\Re_{+}^{n} = \{ (x_{1},\dots,x_{n})^{\top} \in \Re^{n} \mid 
  x_{i} \ge 0 \ (i=1,\dots,n) \}$. 
The $n$-dimensional rotated second-order cone, denoted $\KC^{n}$, is 
defined by 
\begin{align*}
  \KC^{n} = \{
  (\bi{x},y,z) \in \Re^{n-2} \times \Re \times \Re
  \mid
  \bi{x}^{\top} \bi{x} \le y z , \
  y \ge 0 , \
  z \ge 0 
  \} . 
\end{align*}
It is easy to verify that condition $(\bi{x},y,z) \in \KC^{n}$ 
is equivalent to 
\begin{align*}
  y + z \ge 
  \begin{Vmatrix}
    \begin{bmatrix}
      y - z \\
      2\bi{x} \\
    \end{bmatrix}
  \end{Vmatrix}
  , 
\end{align*}
namely, $\KC^{n}$ can be expressed with 
an $n$-dimensional second-order cone constraint.

\section{Truss topology optimization}
\label{sec:truss}

This section presents a convex formulation for truss topology 
optimization under the unilateral contact conditions.

\subsection{Problem setting}

Following the conventional ground structure method, consider an initial 
truss that consists of $m$ members and has $n$ degrees of freedom of 
the nodal displacements. 
Suppose that several nodes of the truss can possibly make contact with 
a rigid obstacle, and that the set of contact candidate nodes is 
specified in advance. 
Throughout the paper, we assume that the obstacle is fixed in space, 
and that contact between each contact candidate node and the obstacle 
surface is frictionless and adhesionless. 
Also, we assume linear elasticity and small deformation. 

Let $\bi{u} \in \Re^{n}$ denote the nodal displacement vector. 
We use $c_{e} \in \Re$ to denote the elongation of 
member $e$ $(e=1,\dots,m)$. 
The compatibility relations can be written in the form 
\begin{align}
  c_{e} = \bi{b}_{e}^{\top} \bi{u} , 
  \quad  e=1,\dots,m, 
  \label{eq:truss.compatibility}
\end{align}
where $\bi{b}_{e} \in \Re^{n}$ is a constant vector. 
Let $l_{e}$ and $E$ denote the initial length of member $e$ and the 
Young modulus, respectively. 
We use $x_{1},\dots,x_{m}$ to denote the member cross-sectional areas, 
which are considered design variables of the topology optimization problem. 
The stiffness matrix of the truss is given by 
\begin{align}
  K(\bi{x}) 
  = \sum_{e=1}^{m} \frac{E}{l_{e}} x_{e} \bi{b}_{e} \bi{b}_{e}^{\top} . 
  \label{eq:truss.stiffness}
\end{align}

Let $c$ denote the number of contact candidate nodes. 
We use $g_{j} \ge 0$ $(j=1,\dots,c)$ to denote the initial gap between 
the $j$th contact candidate node and the obstacle surface. 
The non-penetration conditions for these nodes can be written as 
\begin{align}
  \bi{g} - C_{\rr{n}} \bi{u} \ge \bi{0} , 
  \label{eq:truss.non-penetration}
\end{align}
where $C_{\rr{n}} \in \Re^{c \times n}$ is a constant matrix consisting 
of the unit inner normal vectors of the obstacle surface; 
see, e.g., \cite{MR02,Wri06,Kan11} for fundamentals on kinematics 
in contact mechanics. 

In structural optimization, the compliance is conventionally used as 
a measure of global flexibility of a structure. 
If all the prescribed nodal displacements are equal to zero, then the 
compliance can be defined as the external work done by the prescribed 
external nodal forces. 
Therefore, to maximize the stiffness of a structure, it is often that 
the external work is minimized. 
However, when there exists a non-zero prescribed displacement, it is 
known that the external work is not suitable as a measure of structural 
flexibility and the compliance should be defined by using the total 
potential energy (in the manner described below) \citep{KS12,NXC11,Str12}. 
Since the situation considered in this paper can possibly involves 
non-zero prescribed displacements at the contact candidate nodes 
(unless $\bi{g}\not=\bi{0}$), we adopt the latter definition. 
Let $\bi{f} \in \Re^{n}$ denote the prescribed external nodal load vector. 
For a given truss design $\bi{x} \in \Re_{+}^{m}$, its compliance is 
defined by 
\begin{align}
  \pi(\bi{x}) 
  = \sup_{\bi{u} \in \Re^{n}} 
  \{ 2 \bi{f}^{\top} \bi{u}  - \bi{u}^{\top} K(\bi{x}) \bi{u} 
  \mid
  C_{\rr{n}} \bi{u} \le \bi{g} \} . 
  \label{eq:def.pi.truss}
\end{align}
The truss topology optimization problem is then formulated as follows: 
\begin{subequations}\label{P.truss.design.1}%
  \begin{alignat}{3}
    & \MIN_{\bi{x}}
    &{\quad}& 
    \pi(\bi{x}) \\
    & \ST && 
    \bi{x} \ge \bi{0} , \\
    & && 
    \bi{l}^{\top} \bi{x} \le v . 
  \end{alignat}
\end{subequations}
Here, $v > 0$ is the specified upper bound for the structural volume.

\subsection{Reformulation}

In this section, we show that problem \eqref{P.truss.design.1} can be 
recast as a convex optimization problem. 
A key step for this is formulating the Lagrange dual problem of the 
maximization problem on the right side of \eqref{eq:def.pi.truss}, as 
performed below. 

\begin{theorem}\label{thm:truss.duality}
  For any $\bi{x} \in \Re_{+}^{m}$, 
  $\pi(\bi{x})$ defined by \eqref{eq:def.pi.truss} coincides with 
  the optimal value of the following optimization problem: 
    \begin{alignat*}{3}
      & \MIN
      &{\quad}& 
      \sum_{e=1}^{m} w_{e}
      - 2 \bi{g}^{\top} \bi{r}  \\
      & \ST && 
      w_{e} x_{e} 
      \ge \frac{l_{e}}{E} q_{e}^{2} , 
      \quad e=1,\dots,m, \\
      & && 
      \sum_{e=1}^{m} q_{e} \bi{b}_{e} 
      = \bi{f} + C_{\rr{n}}^{\top} \bi{r} , \\
      & && 
      \bi{r} \le \bi{0} . 
    \end{alignat*}
  Here, $w_{1},\dots,w_{m} \in \Re$, $q_{1},\dots,q_{m} \in \Re$, 
  and $\bi{r} \in \Re^{c}$ are variables to be optimized, 
  and the optimal value is defined to be $+\infty$ if the problem is 
  infeasible.\footnote{%
  Since the maximization problem on the right side of 
  \eqref{eq:def.pi.truss} is feasible, we have 
  $\pi(\bi{x})>-\infty$ $(\forall \bi{x} \in \Re_{+}^{m})$. }
\end{theorem}
\begin{proof}
  We first observe that, by substituting \eqref{eq:truss.compatibility} 
  and \eqref{eq:truss.stiffness} into \eqref{eq:def.pi.truss}, 
  $\pi(\bi{x})$ coincides with the optimal value of the following 
  optimization problem: 
  \begin{subequations}\label{P.truss.dual.compliance.2}%
    \begin{alignat}{3}
      & \MAX
      &{\quad}& 
      2 \bi{f}^{\top} \bi{u} 
      - \sum_{e=1}^{m} \frac{E}{l_{e}} x_{e} c_{e}^{2}  \\
      & \ST && 
      c_{e} = \bi{b}_{e}^{\top} \bi{u} , 
      \quad e=1,\dots,m, \\
      & && 
      C_{\rr{n}} \bi{u} \le \bi{g} . 
    \end{alignat}
  \end{subequations}
  Here, $\bi{u} \in \Re^{n}$ and $\bi{c} \in \Re^{m}$ are variables to 
  be optimized. 
  For notational simplicity, let 
  \begin{align}
    U = \Re^{n} \times \Re^{m} , 
    \quad
    V = \Re^{m} \times \Re^{c} 
    \label{eq:def.spaces}
  \end{align}
  Associated with problem \eqref{P.truss.dual.compliance.2}, the Lagrangian, 
  $L : U \times V \to \overRe$, is defined by 
  \begin{align}
    L(\bi{u},\bi{c}; \bi{q},\bi{r}) 
    &= 
    \begin{dcases*}
      2 \bi{f}^{\top} \bi{u} 
      - \sum_{e=1}^{m} \frac{E}{l_{e}} x_{e} c_{e}^{2} \\
      \qquad 
      + 2 \sum_{e=1}^{m} q_{e} (c_{e} - \bi{b}_{e}^{\top} \bi{u})
      - 2 \bi{r}^{\top} (\bi{g} - C_{\rr{n}} \bi{u}) 
      & if $\bi{r} \le \bi{0}$, \\
      +\infty
      & otherwise, 
    \end{dcases*}
    \label{eq:def.Lagrangian.1}
  \end{align}
  where $\bi{q} \in \Re^{m}$ and $\bi{r} \in \Re^{c}$ are 
  the Lagrange multipliers. 
  Since all the constraints of problem \eqref{P.truss.dual.compliance.2} 
  are affine, the Lagrange duality theory ensures the strong duality 
  holds \citep{ET76,Cia89}: 
  \begin{align}
    \sup_{(\bi{u},\bi{c}) \in U} 
    \inf_{(\bi{q},\bi{r}) \in V} 
    L(\bi{u},\bi{c}; \bi{q},\bi{r}) 
    = \inf_{(\bi{q},\bi{r}) \in V} 
    \sup_{(\bi{u},\bi{c}) \in U} 
    L(\bi{u},\bi{c}; \bi{q},\bi{r}) . 
    \label{eq:duality.varPhi.2}
  \end{align}
  Here, the left side of \eqref{eq:duality.varPhi.2} corresponds to 
  problem \eqref{P.truss.dual.compliance.2}, and the right side is its 
  dual problem. 
  In the remainder, we show that this dual problem is equivalent to 
  the minimization problem stated in this theorem. 
  
  By direct calculations, for $\bi{r} \le \bi{0}$ we obtain 
  \begin{align}
    \MoveEqLeft
    \sup_{(\bi{u},\bi{c}) \in U} 
    L(\bi{u},\bi{c}; \bi{q},\bi{r}) 
    \notag\\
    &= \sup_{\bi{u} \in \Re^{n}} \Bigl\{
    2 \bi{u}^{\top} \Bigl(
    \bi{f} - \sum_{e=1}^{m} q_{e} \bi{b}_{e} 
    + C_{\rr{n}}^{\top} \bi{r} \Bigr) 
    \Bigr\} 
    + \sum_{e=1}^{m} \sup_{c_{e} \in \Re} \Bigl\{
    2 q_{e} c_{e} - \frac{E}{l_{e}} x_{e} c_{e}^{2} 
    \Bigr\}
    - 2 \bi{g}^{\top} \bi{r} 
    \notag\\
    & = 
    \begin{dcases*}
      \sum_{e=1}^{m} \sup_{c_{e} \in \Re} \Bigl\{
      2 q_{e} c_{e} - \frac{E}{l_{e}} x_{e} c_{e}^{2} 
      \Bigr\}
      - 2 \bi{g}^{\top} \bi{r} 
      & if $\displaystyle
      \bi{f} - \sum_{e=1}^{m} q_{e} \bi{b}_{e} 
      + C_{\rr{n}}^{\top}  \bi{r} = \bi{0}$, \\
      +\infty
      & otherwise. 
    \end{dcases*}
    \label{eq:def.Lagrangian.6}
  \end{align}
  We can further reduce the first term of the last expression to 
  \begin{align}
    \MoveEqLeft
    \sup_{c_{e} \in \Re} \Bigl\{
    2 q_{e} c_{e} - \frac{E}{l_{e}} x_{e} c_{e}^{2} 
    \Bigr\}    \notag\\
    &= \min_{w_{e} \in \Re} \Bigl\{
    w_{e}
    \Bigm|
    w_{e} \ge \frac{E}{l_{e}} x_{e} c_{e}^{2} , \ 
    q_{e} = \frac{E}{l_{e}} x_{e} c_{e}
    \Bigr\}  \notag\\
    &= \min_{w_{e} \in \Re} \Bigl\{
    w_{e}
    \Bigm|
    w_{e} x_{e} \ge \frac{E}{l_{e}} x_{e}^{2} c_{e}^{2} , \
    q_{e} = \frac{E}{l_{e}} x_{e} c_{e}
    \Bigr\}  \notag\\
    &= \min_{w_{e} \in \Re} \Bigl\{
    w_{e}
    \Bigm|
    w_{e} x_{e} \ge \frac{l_{e}}{E} q_{e}^{2} 
    \Bigr\} , 
    \label{eq:def.Lagrangian.7}
  \end{align}
  where we have used $x_{e} \ge 0$, $E>0$, and $l_{e}>0$. 
  The proof is completed by 
  substituting \eqref{eq:def.Lagrangian.7} into 
  \eqref{eq:def.Lagrangian.6}. 
\end{proof}

By using \refthm{thm:truss.duality}, we can recast problem 
\eqref{P.truss.design.1} equivalently as follows: 
\begin{subequations}\label{P.truss.SOCP.1}%
  \begin{alignat}{3}
    & \MIN_{\bi{x}, \, \bi{w}, \, \bi{q}, \, \bi{r}}
    &{\quad}& 
    \sum_{e=1}^{m} w_{e}
    - 2 \bi{g}^{\top} \bi{r}  \\
    & \ST && 
    w_{e} x_{e} 
    \ge \frac{l_{e}}{E} q_{e}^{2} , 
    \quad e=1,\dots,m, 
    \label{P.truss.SOCP.1.2} \\
    & && 
    \sum_{e=1}^{m} q_{e} \bi{b}_{e} 
    = \bi{f} + C_{\rr{n}}^{\top} \bi{r} , 
    \label{P.truss.SOCP.1.3} \\
    & && 
    \bi{r} \le \bi{0} , 
    \label{P.truss.SOCP.1.4} \\
    & && 
    \bi{x} \ge \bi{0} , 
    \label{P.truss.SOCP.1.5} \\
    & && 
    \bi{l}^{\top} \bi{x} \le v . 
  \end{alignat}
\end{subequations}
This problem is an SOCP problem; to see this explicitly, rewrite 
\eqref{P.truss.SOCP.1.2} and \eqref{P.truss.SOCP.1.5} as 
\begin{align*}
  \begin{bmatrix}
    \sqrt{l_{e}/E} q_{e} \\
    w_{e} \\
    x_{e} \\
  \end{bmatrix}
  \in \KC^{3}  , 
  \quad e=1,\dots,m.  
\end{align*}
We can solve SOCP problems efficiently with a primal-dual interior-point 
method \cite{AL12}. 


\begin{remark}
  Constraint \eqref{P.truss.SOCP.1.3} can be interpreted as the 
  force-balance equation, where 
  $q_{e}$ $(e=1,\dots,m)$ and $r_{j}$ $(j=1,\dots,c)$ 
  correspond to the axial force of truss member $e$ and 
  the normal contact reaction at contact candidate node $j$, 
  respectively. 
  Constraint \eqref{P.truss.SOCP.1.4} corresponds to the non-adhesion 
  condition of the unilateral contact. 
  \finbox
\end{remark}

\begin{remark}
  Extension of problem \eqref{P.truss.SOCP.1} to a multiple load case 
  is obvious. 
  \finbox
\end{remark}

\begin{remark}
  The form in \eqref{P.truss.dual.compliance.2} is crucial for deriving 
  our convex reformulation in \eqref{P.truss.SOCP.1}. 
  For example, associate with the form on the right side of 
  \eqref{eq:def.pi.truss}, we can define the Lagrangian as 
  \begin{align*}
    \hat{L}(\bi{u}; \bi{r}) = 
    \begin{dcases*}
      2 \bi{f}^{\top} \bi{u} 
      - \bi{u}^{\top} K(\bi{x}) \bi{u} 
      - 2 \bi{r}^{\top} (\bi{g} - C_{\rr{n}} \bi{u}) 
      & if $\bi{r} \le \bi{0}$, \\
      +\infty 
      & otherwise, 
    \end{dcases*}
  \end{align*}
  where $\bi{r} \in \Re^{c}$ is the Lagrange multiplier. 
  Since we have 
  \begin{align*}
    \sup_{\bi{u} \in \Re^{n}} \hat{L}(\bi{u}; \bi{r}) 
    &= \sup_{\bi{u} \in \Re^{n}} \{ 
    2 \bi{u}^{\top} (\bi{f} + C_{\rr{n}}^{\top} \bi{r})
    - \bi{u}^{\top} K(\bi{x}) \bi{u} 
    \} 
    - 2 \bi{g}^{\top} \bi{r}  \notag\\
    &= \{ 
    \bi{u}^{\top} K(\bi{x}) \bi{u} 
    \mid
    (\bi{f} + C_{\rr{n}}^{\top} \bi{r} )
    - K(\bi{x}) \bi{u} = \bi{0} 
    \} 
    - 2 \bi{g}^{\top} \bi{r} 
  \end{align*}
  for any $\bi{r} \le \bi{0}$, the Lagrange dual problem is formulated 
  as follows: 
    \begin{alignat*}{3}
      & \MIN_{\bi{u}, \, \bi{r}}
      &{\quad}& 
      \bi{u}^{\top} K(\bi{x}) \bi{u} - 2 \bi{g}^{\top} \bi{r}  \\
      & \ST && 
      K(\bi{x}) \bi{u} = \bi{f} + C_{\rr{n}}^{\top} \bi{r} , \\
      & && 
      \bi{r} \le \bi{0} . 
    \end{alignat*}
  Therefore, topology optimization problem \eqref{P.truss.design.1} is 
  equivalent also to the following problem in variables 
  $\bi{x} \in \Re^{m}$, $\bi{u} \in \Re^{n}$, and $\bi{r} \in \Re^{c}$: 
  \begin{subequations}\label{P.not.recommended.1}%
    \begin{alignat}{3}
      & \MIN_{\bi{x}, \, \bi{u}, \, \bi{r}}
      &{\quad}& 
      \bi{u}^{\top} K(\bi{x}) \bi{u} - 2 \bi{g}^{\top} \bi{r}  \\
      & \ST && 
      K(\bi{x}) \bi{u} = \bi{f} + C_{\rr{n}}^{\top} \bi{r} , 
      \label{P.not.recommended.1.3} \\
      & && 
      \bi{r} \le \bi{0} , \\
      & && 
      \bi{x} \ge \bi{0} , \\
      & && 
      \bi{l}^{\top} \bi{x} \le v . 
    \end{alignat}
  \end{subequations}
  This problem, however, has a nonconvex objective function and 
  nonconvex equality constraints. 
  Therefore, problem \eqref{P.truss.SOCP.1} is much preferable to 
  problem \eqref{P.not.recommended.1}. 
  Alternatively, if $K(\bi{\rho})$ is regular, then we can eliminate 
  variable $\bi{u}$ from problem \eqref{P.not.recommended.1} by using 
  \eqref{P.not.recommended.1.3}.  
  This yields the following form: 
    \begin{alignat*}{3}
      & \MIN_{\bi{x}, \, \bi{\zeta}, \, \bi{r}}
      &{\quad}& 
      \bi{\zeta}^{\top} K(\bi{x})^{-1} \bi{\zeta} - 2 \bi{g}^{\top} \bi{r} \\
      & \ST && 
      \bi{\zeta} = \bi{f} + C_{\rr{n}}^{\top} \bi{r} , \\
      & && 
      \bi{r} \le \bi{0} , \\
      & && 
       \bi{x} \ge \epsilon \bi{1} , \\
      & && 
      \bi{l}^{\top} \bi{x} \le v . 
    \end{alignat*}
  Here, $\epsilon > 0$ is a small constant to avoid singularity 
  of $K(\bi{x})$. 
  Since the objective function of this problem is nonconvex, 
  problem \eqref{P.truss.SOCP.1} is again much preferable. 
  \finbox
\end{remark}

\subsection{Relation with existing formulations}
\label{sec:truss.relation}

To obtain an optimal solution of problem \eqref{P.truss.design.1}, 
\citet{KZZ98} proposed to solve a convex optimization problem that is 
different from problem \eqref{P.truss.SOCP.1}.\footnote{%
A very similar approach to {\em free-material\/} continuum-based 
topology optimization can be found in \citet{BKNZ00}. } 
Specifically, their approach solves an optimization problem with a 
linear objective function and some convex quadratic constraints. 
This problem can further be reduced to a {\em quadratic programming\/} 
(QP) problem, where a convex quadratic function is minimized under some 
linear inequality constraints. 
A similar formulation was presented also by \citet{KPR95}.\footnote{%
The problem considered by \citet{KPR95} is slightly different from 
problem \eqref{P.truss.design.1}. 
There, the initial gaps are also considered design variables, and the 
normal contact reactions are required to be uniformly distributed. } 
In this section, we establish the relation between this approach and 
our approach, i.e., problem \eqref{P.truss.SOCP.1}. 
The analysis using the perturbation function for the problem under 
consideration attempts to provide deep understanding of the attribute of 
the problem.

For notational simplicity, define $X$ by 
\begin{align*}
  X = \{ \bi{x} \in \Re_{+}^{m} 
  \mid
  \bi{l}^{\top} \bi{x} \le v 
  \} . 
\end{align*}
We also use $U$ and $V$ defined by \eqref{eq:def.spaces}. 
For every $\bi{x} \in X$, define 
$\varPhi_{\bi{x}} : U \times V \to \overRe$ by 
\begin{align}
  \varPhi_{\bi{x}}(\bi{u},\bi{c}; \bi{\eta},\bi{\lambda}) = 
  \begin{dcases*}
    \sum_{e=1}^{m} \frac{E}{l_{e}} x_{e} c_{e}^{2} 
    - 2 \bi{f}^{\top} \bi{u} 
    & if $2(c_{e} - \bi{b}_{e}^{\top} \bi{u}) = \eta_{e}$ 
    $(e=1,\dots,m)$, \\
    & \quad $2(\bi{g} - C_{\rr{n}} \bi{u}) \ge -\bi{\lambda}$, \\
    +\infty 
    & otherwise. 
  \end{dcases*}
  \label{eq:def.varPhi.function}
\end{align}
It is easy to verify that $\varPhi_{\bi{x}}$ is a closed proper convex 
function for any $\bi{x} \in X$. 
By definition in \eqref{eq:def.pi.truss}, the compliance can be written 
as 
\begin{align}
  \pi(\bi{x}) 
  = \sup_{(\bi{u},\bi{c}) \in U} 
  -\varPhi_{\bi{x}}(\bi{u},\bi{c}; \bi{0},\bi{0})  . 
\end{align}
Thus, $\varPhi$ is considered a perturbation function 
associated with the optimization problem on the right side of 
\eqref{eq:def.pi.truss}. 
The topology optimization problem \eqref{P.truss.design.1} can be 
identified with finding $\bi{x}$ at which the value 
\begin{align}
  \inf_{\bi{x} \in X} 
  \sup_{(\bi{u},\bi{c}) \in U} 
  -\varPhi_{\bi{x}}(\bi{u},\bi{c}; \bi{0},\bi{0})  
  \label{eq:inf-sup.form.1}
\end{align}
is attained.\footnote{%
Essentially, formulation \eqref{eq:inf-sup.form.1} is same as the 
saddle-point formulation studied in \citet{Pet96}. }

Since $-\varPhi_{\bi{x}}(\bi{u},\bi{c}; \bi{0},\bi{0})$ is closed, 
convex with respect to $\bi{x}$, and concave with respect 
to $(\bi{u},c)$, 
the standard minimax theorem asserts \citep{Roc70}
\begin{align}
  \inf_{\bi{x} \in X} 
  \sup_{(\bi{u},\bi{c}) \in U} 
  -\varPhi_{\bi{x}}(\bi{u},\bi{c}; \bi{0},\bi{0}) 
  = \sup_{(\bi{u},\bi{c}) \in U} 
  \inf_{\bi{x} \in X} 
  -\varPhi_{\bi{x}}(\bi{u},\bi{c}; \bi{0},\bi{0}) . 
  \label{eq:inf-sup.form.2}
\end{align}
On the other hand, the duality theory asserts \citep{ET76} 
\begin{align}
  \sup_{(\bi{u},\bi{c}) \in U} 
  -\varPhi_{\bi{x}}(\bi{u},\bi{c}; \bi{0},\bi{0}) 
  &= \inf_{(\bi{\eta}^{*},\bi{\lambda}^{*}) \in V} 
  \varPhi_{\bi{x}}^{*}(\bi{0},\bi{0}; \bi{\eta}^{*},\bi{\lambda}^{*}) , 
  \label{eq:duality.varPhi.1}
\end{align}
from which problem \eqref{eq:inf-sup.form.1} can be reduced to 
\begin{align}
  \inf_{\bi{x} \in X} 
  \sup_{(\bi{u},\bi{c}) \in U} 
  -\varPhi_{\bi{x}}(\bi{u},\bi{c}; \bi{0},\bi{0}) 
  = \inf_{(\bi{x},\bi{\eta}^{*},\bi{\lambda}^{*}) \in X \times V} 
  \varPhi_{\bi{x}}^{*}(\bi{0},\bi{0}; \bi{\eta}^{*},\bi{\lambda}^{*}) . 
  \label{eq:duality.varPhi.0}
\end{align}
In a nutshell, the QP approach uses \eqref{eq:inf-sup.form.2}, as well as 
the facts that $\varPhi_{\bi{x}}(\bi{u},\bi{c}; \bi{0},\bi{0})$ is 
linear with respect to $\bi{x}$ and vertices of polyhedron $X$ can be 
obtained explicitly. 
In contrast, our formulation in \eqref{P.truss.SOCP.1} is essentially 
based on \eqref{eq:duality.varPhi.0}. 
In the remainder, we give brief exposition of the two approaches. 

We begin with the QP approach. 
For notational simplicity, define $U_{\rr{F}}$ by 
\begin{align*}
  U_{\rr{F}} 
  = \{ (\bi{u},\bi{c}) \in U 
  \mid
  c_{e} = \bi{b}_{e}^{\top} \bi{u} 
  \ (e=1,\dots,m), \
  C_{\rr{n}} \bi{u} \le \bi{g} 
  \} . 
\end{align*}
By direct calculations, the right side of \eqref{eq:inf-sup.form.2} is 
reduced to 
\begin{align}
  \MoveEqLeft
  \sup_{(\bi{u},\bi{c}) \in V} 
  \inf_{\bi{x} \in X} 
  -\varPhi_{\bi{x}}(\bi{u},\bi{c}; \bi{0},\bi{0}) 
  \notag\\
  &= 
  \sup_{(\bi{u},\bi{c}) \in U_{\rr{F}}} 
  \inf_{\bi{x} \in X} \Bigl\{ 
  2 \bi{f}^{\top} \bi{u} 
  - \sum_{e=1}^{m} \frac{E}{l_{e}} x_{e} c_{e}^{2}
  \Bigr\}
  \notag\\
  &= 
  \sup_{(\bi{u},\bi{c}) \in U_{\rr{F}}} \left\{
  2 \bi{f}^{\top} \bi{u} 
  - \sup_{\bi{x} \in X } 
  \Bigl\{ \sum_{e=1}^{m} \frac{E}{l_{e}} x_{e} c_{e}^{2}
  \Bigr\}
  \right\} . 
  \label{eq:QP.1}
\end{align}
From the standard theory of linear programming, we obtain 
\begin{align}
  \sup_{\bi{x} \in X } 
  \Bigl\{ \sum_{e=1}^{m} \frac{E}{l_{e}} x_{e} c_{e}^{2}
  \Bigr\}
  = \max
  \Bigl\{ 
  \frac{E v}{l_{1}^{2}} c_{1}^{2}, \dots, \frac{E v}{l_{m}^{2}} c_{m}^{2}
  \Bigr\}  , 
  \label{eq:QP.2}
\end{align}
where the vertices of $X$ have been considered to evaluate the optimal 
value of a linear programming problem. 
Substitution of \eqref{eq:QP.2} into \eqref{eq:QP.1} yields the 
following optimization problem: 
\begin{subequations}\label{P.truss.LP.1}%
  \begin{alignat}{3}
    & \MAX_{\bi{u}, \, \bi{c}, \, \alpha}
    &{\quad}& 
    2 \bi{f}^{\top} \bi{u} - \alpha \\
    & \ST && 
    \alpha \ge \frac{E v}{l_{e}^{2}} c_{e}^{2} , 
    \quad e=1,\dots,m, \\
    & && 
    c_{e} = \bi{b}_{e}^{\top} \bi{u} , 
    \quad e=1,\dots,m, \\
    & && 
    C_{\rr{n}} \bi{u} \le \bi{g} . 
  \end{alignat}
\end{subequations}
This is minimization of a linear objective function under the convex 
quadratic constraints, and can be recast as an SOCP problem. 
Problem \eqref{P.truss.LP.1} can further be reduced to the following form: 
\begin{subequations}\label{P.truss.LP.4}%
  \begin{alignat}{3}
    & \MAX_{\bi{u}, \, \bi{c}, \, \beta}
    &{\quad}& 
    2 \bi{f}^{\top} \bi{u} - \beta^{2} \\
    & \ST && 
    \beta \ge \frac{\sqrt{E v}}{l_{e}} c_{e} \ge -\beta , 
    \quad e=1,\dots,m, \\
    & && 
    c_{e} = \bi{b}_{e}^{\top} \bi{u} , 
    \quad e=1,\dots,m, 
    \label{P.truss.LP.4.3} \\
    & && 
    C_{\rr{n}} \bi{u} \le \bi{g} . 
  \end{alignat}
\end{subequations}
This is a QP problem (i.e., by converting 
maximization to minimization, we obtain a problem that minimizes a 
convex quadratic function under some linear constraints).\footnote{%
Clearly, we can eliminate variable $\bi{c}$ from problem 
\eqref{P.truss.LP.4} by using \eqref{P.truss.LP.4.3}. 
However, just for consistency with other formulations in the paper, we do 
not carry out this elimination. 
Similarly, variable $\bi{c}$ in problem \eqref{P.truss.LP.1} can also be 
eliminated, as is done in \cite{KZZ98}. } 

\begin{remark}
  Based on \eqref{eq:inf-sup.form.2}, \citet{PP97} proposed an 
  alternative approach. 
  Define $\psi : \Re^{n} \to \overRe$ by
  \begin{align*}
    \psi(\bi{u}) 
    = \inf_{\bi{x} \in X} 
    -\varPhi(\bi{x}; \bi{u},\bi{c}(\bi{u}); \bi{0},\bi{0}) , 
  \end{align*}
  where 
  \begin{align}
    c_{e}(\bi{u}) = \bi{b}_{e}^{\top} \bi{u} , 
    \quad
    e=1,\dots,m . 
    \label{eq:c.as.a.function}
  \end{align}
  Then the right side of \eqref{eq:inf-sup.form.2} can be written as  
  \begin{align*}
    \sup_{\bi{u} \in \Re^{n}} \psi(\bi{u}) . 
  \end{align*}
  \citet{PP97} applied a subgradient method to this optimization problem. 
  \finbox
\end{remark}

We next see that the approach presented in this paper is essentially 
based on \eqref{eq:duality.varPhi.1}. 
For this purpose, we show that \eqref{eq:duality.varPhi.2}, used in the proof 
of \refthm{thm:truss.duality}, implies \eqref{eq:duality.varPhi.1}. 
From definition \eqref{eq:def.varPhi.function} of $\varPhi_{\bi{x}}$ 
and definition \eqref{eq:def.Lagrangian.1} of $L$, we have 
\begin{align*}
  \MoveEqLeft
  \sup_{(\bi{\eta},\bi{\lambda}) \in V} 
  \left\{ 
  \begin{bmatrix}
    \bi{\eta}^{*} \\    \bi{\lambda}^{*} \\
  \end{bmatrix}
  ^{\top}
  \begin{bmatrix}
    \bi{\eta} \\    \bi{\lambda} \\
  \end{bmatrix}
  - \varPhi_{\bi{x}}(\bi{u},\bi{c}; \bi{\eta},\bi{\lambda})  
  \right\}    \notag\\
  &= 2 \bi{f}^{\top} \bi{u} 
  - \sum_{e=1}^{m} \frac{E}{l_{e}} x_{e} c_{e}^{2} 
  + \sup 
  \left\{ \left.
  \begin{bmatrix}
    \bi{\eta}^{*} \\    \bi{\lambda}^{*} \\
  \end{bmatrix}
  ^{\top}
  \begin{bmatrix}
    \bi{\eta} \\    \bi{\lambda} \\
  \end{bmatrix}
  \ \right|\ 
  (\bi{\eta}, \, \bi{\lambda}) \in 
  \domain \varPhi_{\bi{x}}(\bi{u},\bi{c}; \,\cdot\, , \,\cdot\, )
  \right\}
  \notag\\
  &= 2 \bi{f}^{\top} \bi{u} 
  - \sum_{e=1}^{m} \frac{E}{l_{e}} x_{e} c_{e}^{2} 
  + \sum_{e=1}^{m} \sup_{\eta_{e} \in \Re} \{ 
  \eta_{e}^{*} \eta_{e} 
  \mid
  2(c_{e} - \bi{b}_{e}^{\top} \bi{u}) = \eta_{e} 
  \}    \notag\\
  & \qquad
  {}+ \sup_{\bi{\lambda} \in \Re^{c}} \{
  \bi{\lambda}^{*\top} (2\bi{g} - 2C_{\rr{n}} \bi{u} + \bi{\lambda})
  \mid
  2\bi{g} - 2C_{\rr{n}} \bi{u} + \bi{\lambda} \ge \bi{0} 
  \}
  - \bi{\lambda}^{*\top} (2\bi{g} - 2C_{\rr{n}} \bi{u}) 
  \notag\\
  &= L(\bi{u},\bi{c}; \bi{\eta}^{*},\bi{\lambda}^{*})  . 
\end{align*}
This relation is the key to showing 
\begin{align}
  \varPhi_{\bi{x}}(\bi{u},\bi{c}; \bi{0},\bi{0})  
  &= -\inf_{(\bi{\eta}^{*},\bi{\lambda}^{*}) \in V} 
  L(\bi{u},\bi{c}; \bi{\eta}^{*},\bi{\lambda}^{*}) , 
  \label{eq:duality.varPhi.4} \\
  \varPhi_{\bi{x}}^{*}(\bi{0},\bi{0}; \bi{\eta}^{*},\bi{\lambda}^{*}) 
  &= \sup_{(\bi{u},\bi{c}) \in U} 
  L(\bi{u},\bi{c}; \bi{\eta}^{*},\bi{\lambda}^{*}) , 
  \label{eq:duality.varPhi.5}
\end{align}
which imply \eqref{eq:duality.varPhi.1}. 
Namely, for any 
$\bi{x} \in X$ and any $(\bi{u},\bi{c})\in U$, 
since $\varPhi_{\bi{x}}(\bi{u},\bi{c}; \,\cdot\, , \,\cdot\, )$ 
is a closed proper convex function, 
its biconjugate function coincides with itself \cite{ET76}, i.e., 
\begin{align*}
  \varPhi_{\bi{x}}(\bi{u},\bi{c}; \bi{\eta},\bi{\lambda}) 
  &= \sup_{(\bi{\eta}^{*},\bi{\lambda}^{*}) \in V}
  \left\{
  \begin{bmatrix}
    \bi{\eta}^{*} \\   \bi{\lambda}^{*} \\
  \end{bmatrix}
  ^{\top}
  \begin{bmatrix}
    \bi{\eta} \\   \bi{\lambda} \\
  \end{bmatrix}
  - \sup_{(\bi{\eta},\bi{\lambda}) \in V} 
  \left\{
  \begin{bmatrix}
    \bi{\eta}^{*} \\   \bi{\lambda}^{*} \\
  \end{bmatrix}
  ^{\top}
  \begin{bmatrix}
    \bi{\eta} \\   \bi{\lambda} \\
  \end{bmatrix}
  - \varPhi_{\bi{x}}(\bi{u},\bi{c};  \bi{\eta},\bi{\lambda}) 
  \right\}
  \right\}
  \notag\\
  &= \sup_{(\bi{\eta}^{*},\bi{\lambda}^{*}) \in V}
  \left\{
  \begin{bmatrix}
    \bi{\eta}^{*} \\   \bi{\lambda}^{*} \\
  \end{bmatrix}
  ^{\top}
  \begin{bmatrix}
    \bi{\eta} \\   \bi{\lambda} \\
  \end{bmatrix}
  - L(\bi{u},\bi{c}; \bi{\eta}^{*},\bi{\lambda}^{*}) 
  \right\} , 
\end{align*}
which establishes \eqref{eq:duality.varPhi.4}. 
Moreover, application of the Legendre--Fenchel transform to 
$\varPhi_{\bi{x}}$ yields 
\begin{align*}
  \varPhi_{\bi{x}}^{*}(\bi{u}^{*},\bi{c}^{*}; \bi{\eta}^{*},\bi{\lambda}^{*}) 
  &= \sup_{(\bi{u},\bi{c}) \in U} 
  \left\{
  \begin{bmatrix}
    \bi{u}^{*} \\   \bi{c}^{*} \\
  \end{bmatrix}
  ^{\top}
  \begin{bmatrix}
    \bi{u} \\   \bi{c} \\
  \end{bmatrix}
  + \sup_{(\bi{\eta},\bi{\lambda}) \in V} 
  \left\{
  \begin{bmatrix}
    \bi{\eta}^{*} \\   \bi{\lambda}^{*} \\
  \end{bmatrix}
  ^{\top}
  \begin{bmatrix}
    \bi{\eta} \\   \bi{\lambda} \\
  \end{bmatrix}
  - \varPhi_{\bi{x}}(\bi{u},\bi{c}; \bi{\eta},\bi{\lambda}) 
  \right\}
  \right\}
  \notag\\
  &= \sup_{(\bi{u},\bi{c}) \in U} 
  \left\{
  \begin{bmatrix}
    \bi{u}^{*} \\   \bi{c}^{*} \\
  \end{bmatrix}
  ^{\top}
  \begin{bmatrix}
    \bi{u} \\   \bi{c} \\
  \end{bmatrix}
  + L(\bi{u},\bi{c}; \bi{\eta}^{*},\bi{\lambda}^{*}) 
  \right\} , 
\end{align*}
which establishes \eqref{eq:duality.varPhi.5}. 

Thus, the difference between the QP approach in literature and the 
approach presented in this paper can be clearly captured in 
\eqref{eq:inf-sup.form.2} and \eqref{eq:duality.varPhi.0}. 
One of advantages of the proposed formulation is that it retains the 
design variable $\bi{x}$ and hence we can incorporate various 
practicality constraints on truss design.\footnote{%
Concrete examples of such constraints appear in section~\ref{sec:ex.1}. } 
In contrast, it is very difficult for the QP approach to handle 
constraints concerning truss design, because in this approach the design 
variable $\bi{x}$ is eliminated; see \eqref{eq:QP.2}.

\begin{remark}\label{rem:nonconvex.case}
  It is worth of noting that \eqref{eq:duality.varPhi.0} holds even if 
  $\varPhi_{\bi{x}}(\bi{u},\bi{c}; \bi{\eta},\bi{\lambda})$ is not 
  convex with respect to $\bi{x}$. 
  This allows extending the presented approach to continua, 
  as seen in section~\ref{sec:body}. 
  \finbox
\end{remark}

\begin{remark}
  For problem \eqref{eq:inf-sup.form.1}, approaches based on MPCC for 
  problem can also be found in literature \citep{Tin99,Hil00,SK10,HKP99}. 
  In these approaches, the KKT condition for the inner maximization 
  problem of \eqref{eq:inf-sup.form.1} is treated as constraint for the 
  outer minimization problem. 
  \finbox
\end{remark}

\section{Continuum topology optimization}
\label{sec:body}

This section extends the formulation in section~\ref{sec:truss} to the 
continuum-based topology optimization problem. 
We adopt the SIMP approach \citep{BS99} 
with the density filter \citep{BT01,Bou01}. 

For simplicity, suppose that a planar design domain is discretized as 
regular mesh with square four-node quadrilateral (Q4) finite elements. 
Extensions to a three-dimensional case and other finite elements are 
straightforward. 

Any notation unexplained in this section is the same as the one in 
section~\ref{sec:truss}. 

Let $\bi{u} \in \Re^{\bar{n}}$ and $\bi{f} \in \Re^{\bar{n}}$ denote the 
nodal displacement vector and the external nodal load vector, 
respectively, where $\bar{n}$ is twice the number of nodes. 
The displacement boundary conditions (i.e., the Dirichlet boundary 
conditions) are written in the form 
\begin{align*}
  D \bi{u} = \bi{0} , 
\end{align*}
where $D \in \Re^{d \times \bar{n}}$ is a constant matrix. 
Note that the number of degrees of freedom of the nodal displacements 
is $\bar{n}-d$. 
The non-penetration conditions of the contact candidate nodes are given 
as \eqref{eq:truss.non-penetration} with 
$C_{\rr{n}} \in \Re^{c \times \bar{n}}$. 

Let $m$ denote the number of finite elements. 
We use $\rho_{e} \in [0,1]$ $(e=1,\dots,m)$ to denote the density of 
element $e$, where $\rho_{e}=1$ means that element $e$ exists and 
$\rho_{e}=0$ means that it is absent. 
The stiffness matrix, denoted $K(\bi{\rho}) \in \SC^{\bar{n}}$, is given 
in the form 
\begin{align}
  K(\bi{\rho}) = \sum_{e=1}^{m} \rho_{e}^{p} K_{e} , 
  \label{eq:continuum.stiffness.1}
\end{align}
where $p>1$ is the penalization power of SIMP, 
and $K_{e} \in \SC^{\bar{n}}$ is a constant positive semidefinite matrix. 
Diagonalization of $K_{e}$ results in the form 
\begin{align}
  K_{e} = B_{e} \kappa B_{e}^{\top} , 
  \label{eq:continuum.stiffness.2}
\end{align}
where $\kappa \in \SC^{5}$ is a positive definite constant matrix, 
and $B_{e} \in \Re^{\bar{n} \times 5}$ is a constant matrix. 
Note that $\kappa$ is identical for all finite elements because we 
suppose a regular mesh. 
The compliance is defined by 
\begin{align}
  \pi(\bi{\rho}) 
  = \sup_{\bi{u} \in \Re^{\bar{n}}} \{
  2 \bi{f}^{\top} \bi{u}  - \bi{u}^{\top} K(\bi{\rho}) \bi{u} 
  \mid
  D \bi{u} = \bi{0} , \
  C_{\rr{n}} \bi{u} \le \bi{g} 
  \} . 
  \label{eq:compliance.body}
\end{align}
The following theorem provides another expression of $\pi(\bi{\rho})$, 
that is useful in topology optimization. 

\begin{theorem}\label{thm:continuum.duality}%
  For any $\bi{\rho} \in \Re_{+}^{m}$, 
  $\pi(\bi{\rho})$ in \eqref{eq:compliance.body} coincides with the 
  optimal value of the following optimization problem: 
  \begin{subequations}\label{P.continuum.dual.compliance.1}%
    \begin{alignat}{3}
      & \MIN_{\bi{w}, \, \bi{s}, \, \bi{t}, \, \bi{r}}
      &{\quad}& 
      \sum_{e=1}^{m} w_{e} - 2 \bi{g}^{\top} \bi{r}  \\
      & \ST && 
      w_{e} \rho_{e}^{p} 
      \ge \bi{s}_{e}^{\top} \kappa^{-1} \bi{s}_{e} , 
      \quad e=1,\dots,m, \\
      & && 
      \sum_{e=1}^{m} B_{e} \bi{s}_{e} 
      = \bi{f} + D^{\top} \bi{t} + C_{\rr{n}}^{\top} \bi{r} , 
      \label{P.continuum.dual.compliance.1.3} \\
      & && 
      \bi{r} \le \bi{0} . 
      \label{P.continuum.dual.compliance.1.4}
    \end{alignat}
  \end{subequations}
  Here, if the problem is infeasible, then the optimal value is defined 
  to be $+\infty$. 
\end{theorem}
\begin{proof}
  By substituting \eqref{eq:continuum.stiffness.1} and 
  \eqref{eq:continuum.stiffness.2} into \eqref{eq:compliance.body}, 
  we see that $\pi(\bi{\rho})$ is the optimal value of the following 
  optimization problem: 
  \begin{subequations}\label{P.def.complementary.1}%
    \begin{alignat}{3}
      & \MAX_{\bi{u}, \, \bi{c}}
      &{\quad}& 
      2 \bi{f}^{\top} \bi{u} 
      - \sum_{e=1}^{m} \rho_{e}^{p} (\bi{c}_{e}^{\top} \kappa \bi{c}_{e}) \\
      & \ST && 
      \bi{c}_{e} = B_{e}^{\top} \bi{u} , 
      \quad e=1,\dots,m, 
      \label{P.def.complementary.1.2} \\
      & && 
      \bi{D} \bi{u} = \bi{0} , \\
      & && 
      C_{\rr{n}} \bi{u} \le \bi{g} . 
    \end{alignat}
  \end{subequations}
  Here, optimization variables are 
  $\bi{c}_{1},\dots,\bi{c}_{m} \in \Re^{5}$ and 
  $\bi{u} \in \Re^{\bar{n}}$. 
  Analogous to \refthm{thm:truss.duality}, the assertion of this theorem 
  can be obtained as the Lagrange dual problem of 
  \eqref{P.def.complementary.1}. 
\end{proof}

\begin{remark}
  Constraint \eqref{P.continuum.dual.compliance.1.3} is the 
  force-balance equation, where $\bi{s}_{e} \in \Re^{5}$ 
  is the generalized stress, 
  $\bi{t} \in \Re^{d}$ is the reaction vector stemming from the 
  the Dirichlet boundary conditions, 
  and $\bi{r} \in \Re^{c}$ is the normal contact reaction vector. 
  It is worth noting that $\bi{s}_{e}$ is work-conjugate to the 
  generalized strain $\bi{c}_{e}$ defined by 
  \eqref{P.def.complementary.1.2}. 
  Constraint \eqref{P.continuum.dual.compliance.1.4} corresponds to 
  the non-penetration condition of the unilateral contact law. 
  \finbox
\end{remark}

We are now in position to formulate the topology optimization problem. 
Application of the density filter can be written in the form 
\begin{align*}
  \bi{\rho} = H \bi{x} , 
\end{align*}
where $H \in \Re^{m \times m}$ is a constant matrix, 
$x_{e} \in [0,1]$ is the original density, 
and $\rho_{e}$ is the so-called physical density. 
The topology optimization problem minimizes the compliance evaluated at 
$\bi{\rho}$ under the volume constraint: 
\begin{subequations}\label{P.topology.1}%
  \begin{alignat}{3}
    & \MIN_{\bi{\rho}, \, \bi{x}}
    &{\quad}& 
    \pi(\bi{\rho})  \\
    & \ST && 
    \bi{\rho} = H \bi{x} , \\
    & &&
    \bi{0} \le \bi{x} \le \bi{1} , \\
    & && 
    \bi{1}^{\top} \bi{\rho} \le v . 
  \end{alignat}
\end{subequations}

Application of \refthm{thm:continuum.duality} to problem 
\eqref{P.topology.1} yields the following form: 
\begin{subequations}\label{P.contact.topology.2}%
  \begin{alignat}{3}
    & \MIN_{\bi{\rho}, \, \bi{x}, \, \bi{w}, \, \bi{s}, \, \bi{t}, \, \bi{r}}
    &{\quad}& 
    \sum_{e=1}^{m} w_{e}
    - 2 \bi{g}^{\top} \bi{r}  \\
    & \ST && 
    w_{e} \rho_{e}^{p} 
    \ge \bi{s}_{e}^{\top} \kappa^{-1} \bi{s}_{e} , 
    \quad e=1,\dots,m, 
    \label{P.contact.topology.2.2} \\
    & && 
    \sum_{e=1}^{m} B_{e} \bi{s}_{e}
    = \bi{f} + D^{\top} \bi{t} + C_{\rr{n}}^{\top} \bi{r} , \\
    & && 
    \bi{r} \le \bi{0} , \\
    & && 
    \bi{\rho} = H \bi{x} , \\
    & && 
    \bi{0} \le \bi{x} \le \bi{1} , \\
    & && 
    \bi{1}^{\top} \bi{\rho} \le v . 
  \end{alignat}
\end{subequations}
Here, $\bi{\rho} \in \Re^{m}$, $\bi{x} \in \Re^{m}$, 
$\bi{w} \in \Re^{m}$, $\bi{s}_{e} \in \Re^{5}$ $(e=1,\dots,m)$, 
$\bi{t} \in \Re^{d}$, and $\bi{r} \in \Re^{c}$ are variables to be 
optimized. 
It is worth noting that an artificial small positive lower bound for 
$x_{e}$, which is usually used in topology optimization to avoid 
singularity of the stiffness matrix, is unnecessary for this formulation. 

In problem \eqref{P.contact.topology.2}, only constraint 
\eqref{P.contact.topology.2.2} is nonconvex (due to the SIMP 
penalization). 
To solve problem \eqref{P.contact.topology.2}, we sequentially solve 
SOCP problems that approximate \eqref{P.contact.topology.2}, 
in the same fashion as the sequential 
{\em semidefinite programming\/} (SDP) for 
nonlinear SDP problems \citep{KT06,KNKF05}. 
Let $\rho_{e}^{(k)}$ denote the incumbent value of variable $\rho_{e}$. 
We linearize $\rho_{e}^{p}$ in \eqref{P.contact.topology.2.2} at 
$\rho_{e}^{(k)}$ as 
\begin{align*}
  \rho_{e}^{p} 
  \simeq 
   p (\rho_{e}^{(k)})^{p-1} \rho_{e} 
  + (1-p) (\rho_{e}^{(k)})^{p} . 
\end{align*}
Then constraint \eqref{P.contact.topology.2.2} is approximated as 
\begin{align*}
  \begin{bmatrix}
    \kappa^{-1/2} \bi{s}_{e} \\
    w_{e} \\
    p (\rho_{e}^{(k)})^{p-1} \rho_{e}
    + (1-p) (\rho_{e}^{(k)})^{p} \\
  \end{bmatrix}
  \in \KC^{7} , 
\end{align*}
where $\kappa^{-1/2}$ is the symmetric square root of $\kappa^{-1}$, 
i.e., $\kappa^{-1/2} \in \SC^{5}$ and 
$\kappa^{-1/2} \kappa^{-1/2} = \kappa^{-1}$. 
Thus, we can construct an SOCP subproblem that approximates 
problem \eqref{P.contact.topology.2}.

\section{Numerical examples}
\label{sec:ex}

This section presents some numerical examples for the approaches 
proposed in section~\ref{sec:truss} and section~\ref{sec:body}. 
Numerical experiments were carried out on 
a $2.2\,\mathrm{GHz}$ Intel Core i5 processor with $8\,\mathrm{GB}$ RAM. 

\subsection{Example (I): Trusses}
\label{sec:ex.1}

\begin{figure}[tbp]
  \centering
  \includegraphics[scale=0.35]{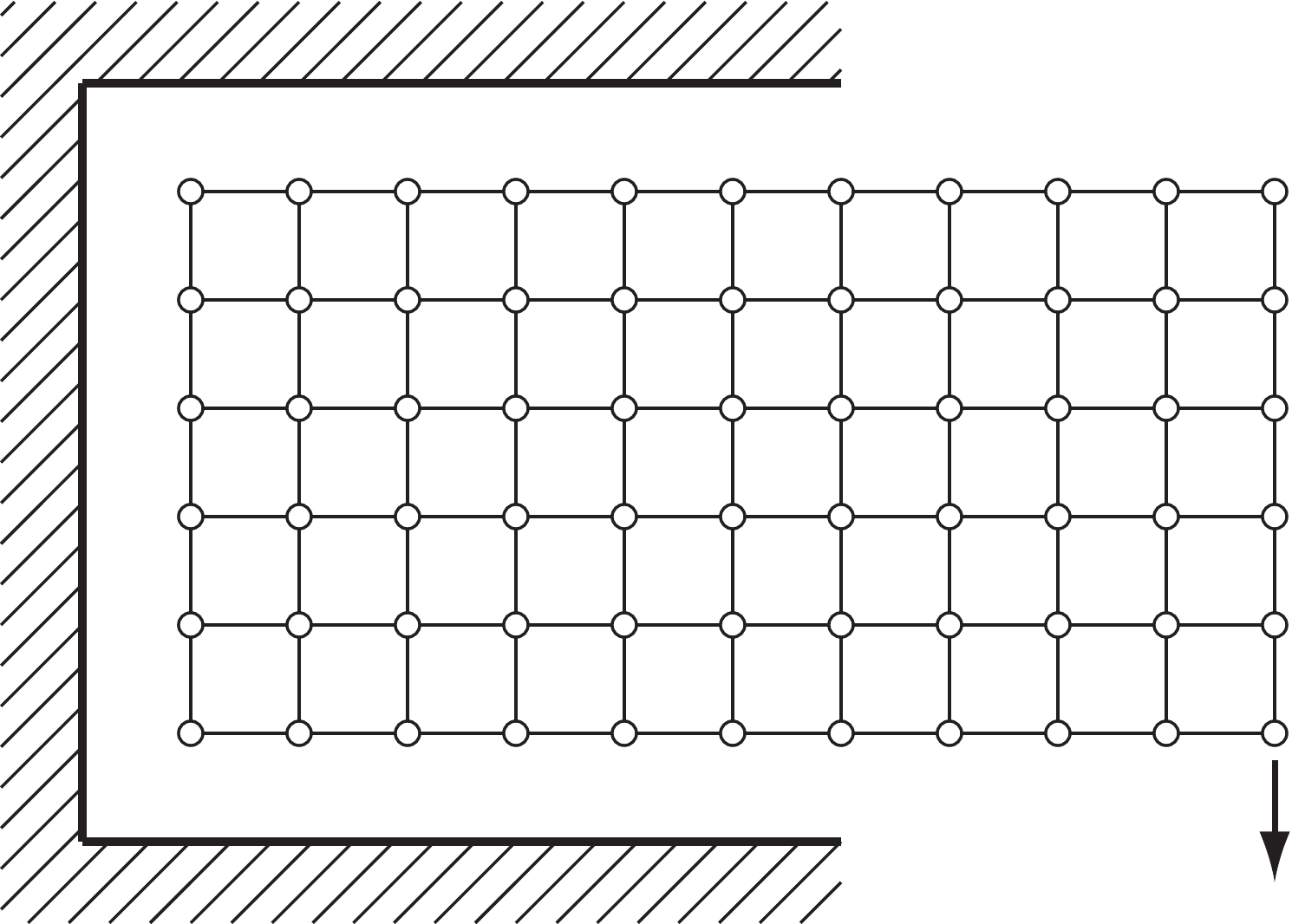}
  \caption{Problem setting of example (I). }
  \label{fig:initial_truss}
\end{figure}

\begin{figure}[tbp]
  \centering
  \subfloat[]{
  \label{fig:x10_y5_gap0_downward_cont}
  \includegraphics[scale=0.45]{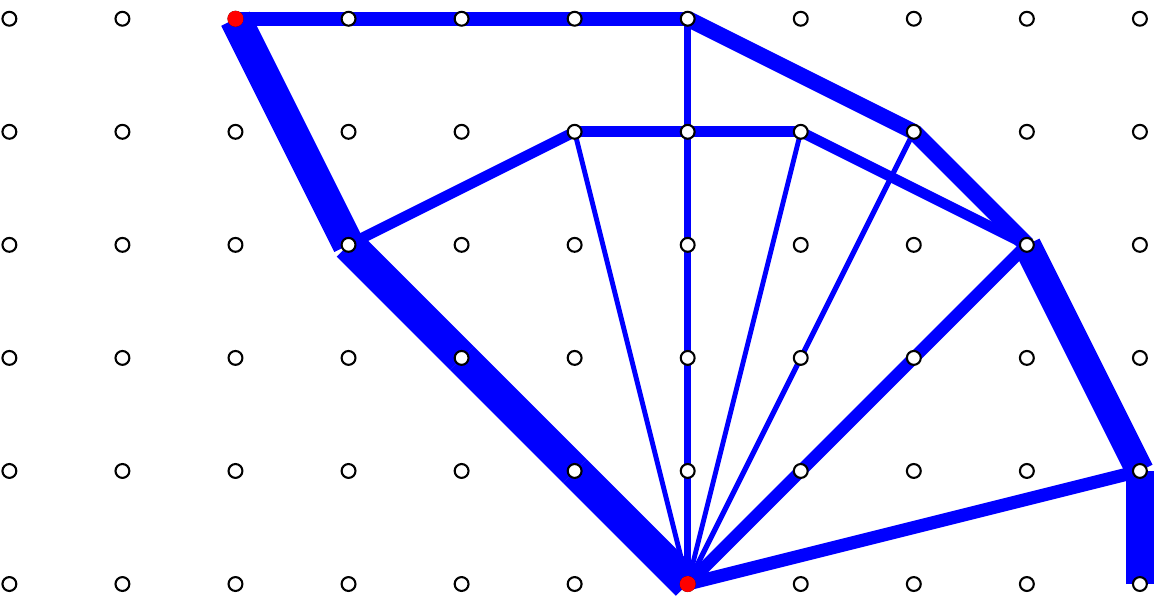}
  }
  \quad
  \subfloat[]{
  \label{fig:x10_y5_gap1_downward_cont}
  \includegraphics[scale=0.45]{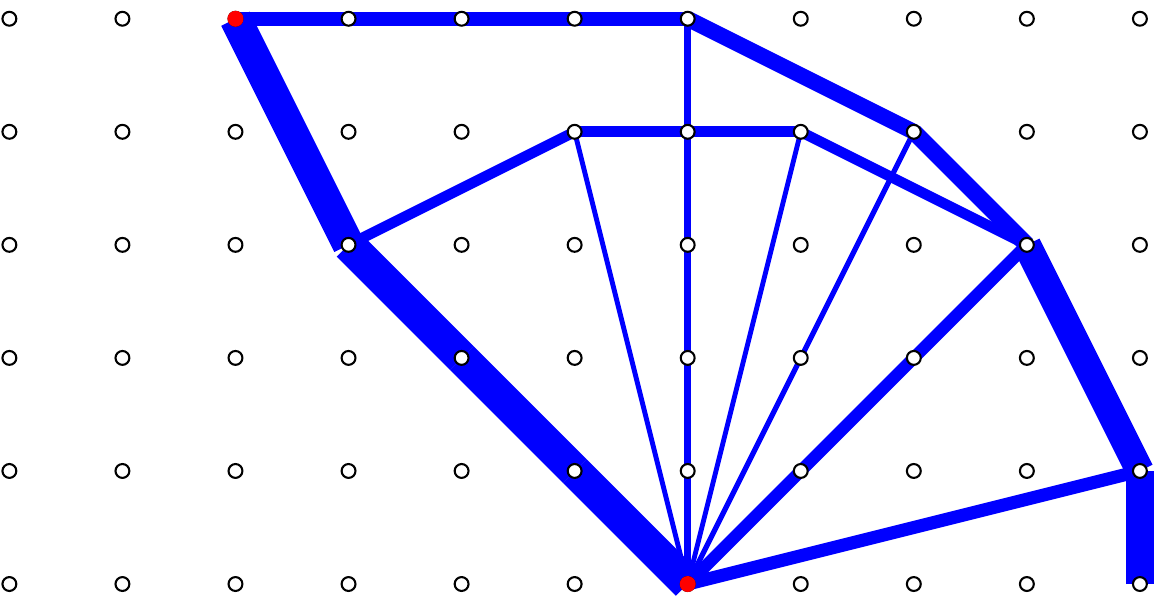}
  }
  \par
  \subfloat[]{
  \label{fig:x10_y5_gap2_downward_cont}
  \includegraphics[scale=0.45]{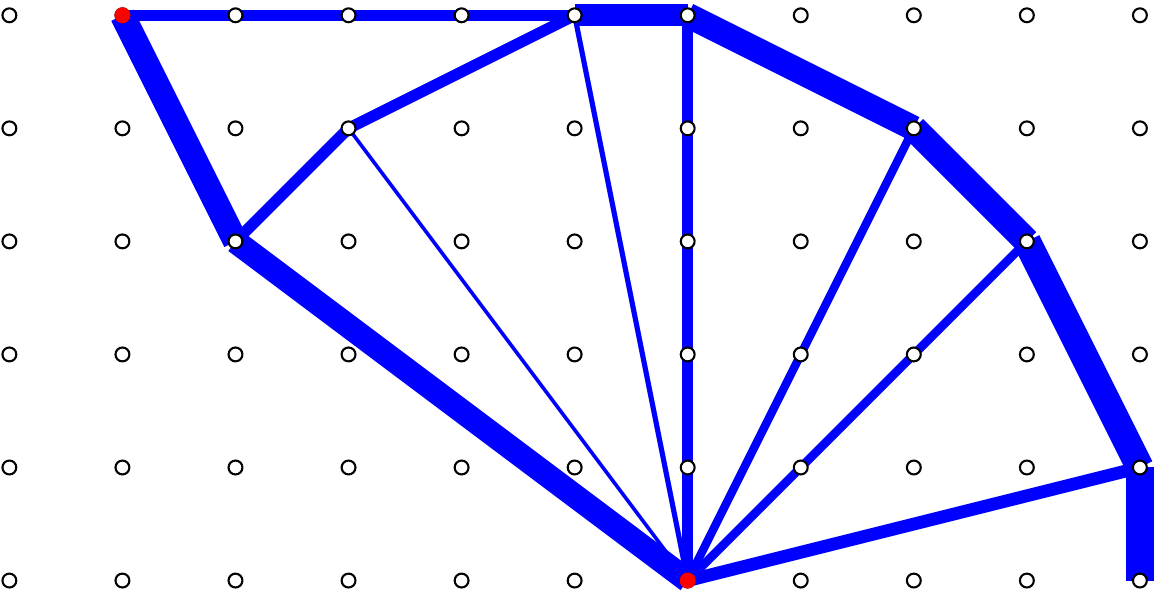}
  }
  \quad
  \subfloat[]{
  \label{fig:x10_y5_gap3_downward_cont}
  \includegraphics[scale=0.45]{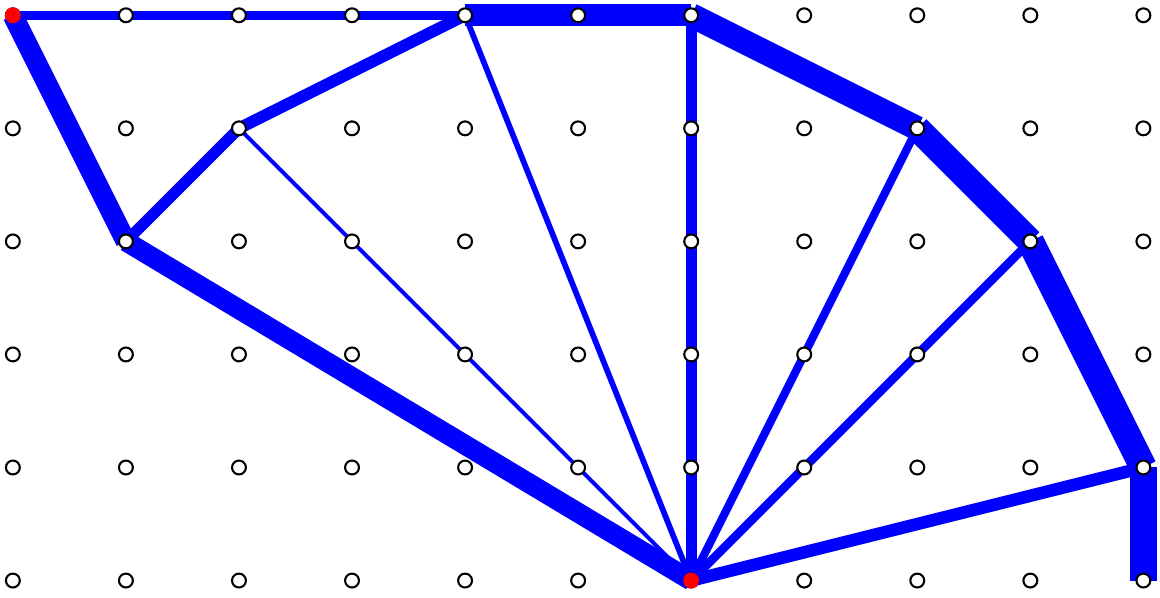}
  }
  \caption[]{The optimal solutions of problem \eqref{P.truss.SOCP.1} 
  (example (I)). 
  The initial gaps are 
  \subref{fig:x10_y5_gap0_downward_cont}~$g_{j}=0$; 
  \subref{fig:x10_y5_gap1_downward_cont}~$g_{j}=0.25\,\mathrm{mm}$; 
  \subref{fig:x10_y5_gap2_downward_cont}~$g_{j}=0.5\,\mathrm{mm}$; and 
  \subref{fig:x10_y5_gap3_downward_cont}~$g_{j}=0.75\,\mathrm{mm}$. 
  }
  \label{fig:x10_y5_gap*_downward_cont}
\end{figure}

\begin{figure}[tbp]
  \centering
  \includegraphics[scale=0.45]{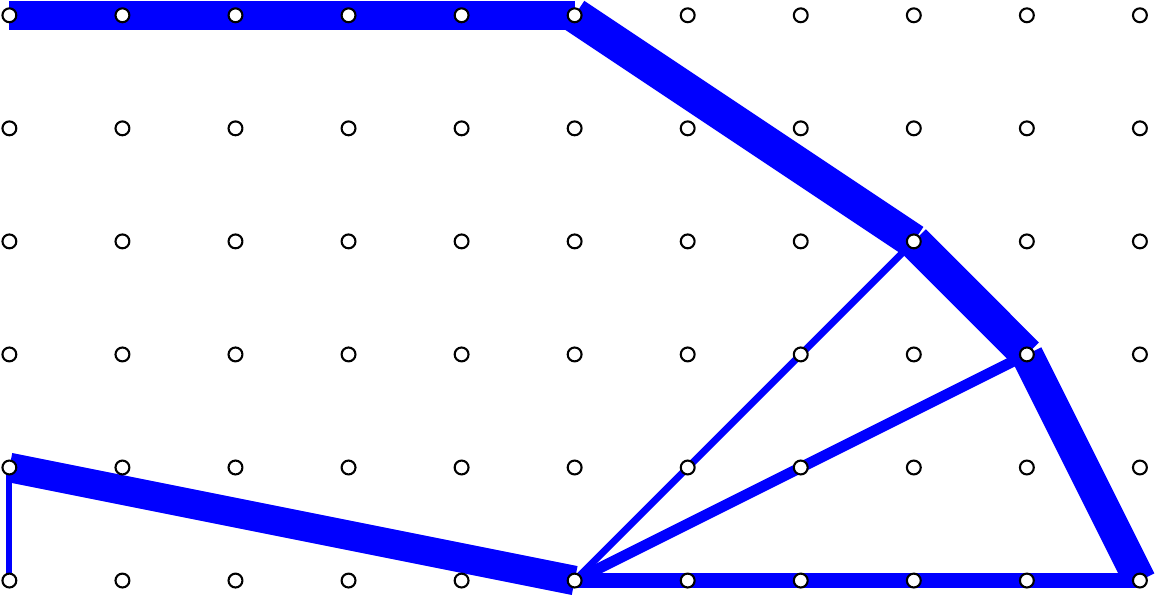}
  \caption{The optimal solution with bilateral contact conditions in 
  example (I). The initial gaps are $g_{j}=0$. }
  \label{fig:x10_y5_gap0_downward_bi_cont}
\end{figure}

\begin{figure}[tbp]
  \centering
  \subfloat[]{
  \label{fig:x10_y5_gap0_downward_no_slen_deg5}
  \includegraphics[scale=0.45]{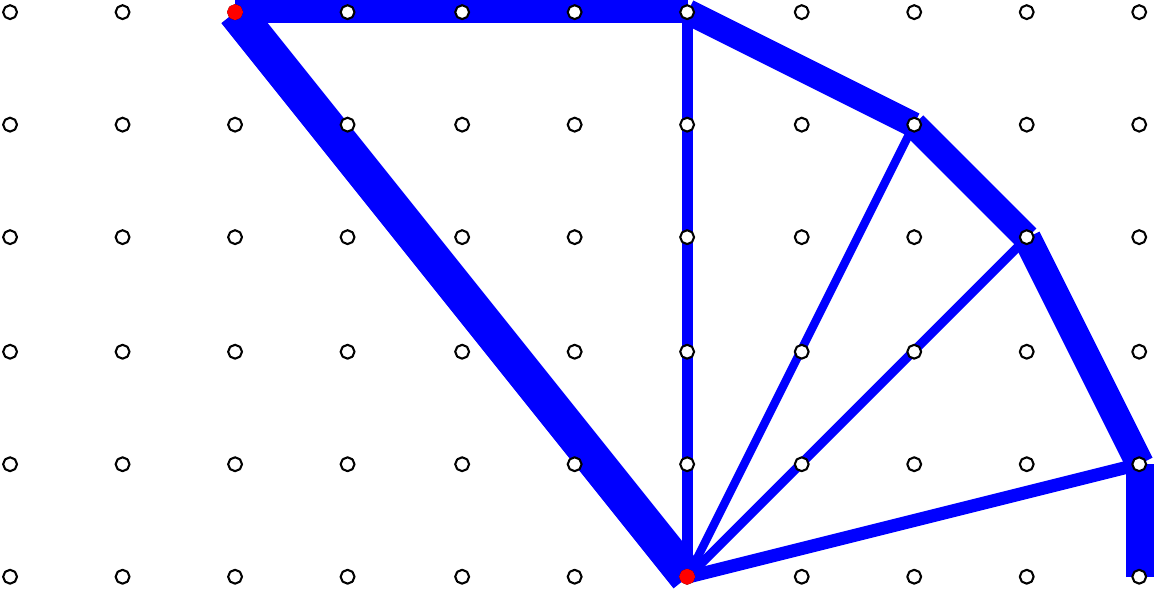}
  }
  \quad
  \subfloat[]{
  \label{fig:x10_y5_gap1_downward_no_slen_deg5}
  \includegraphics[scale=0.45]{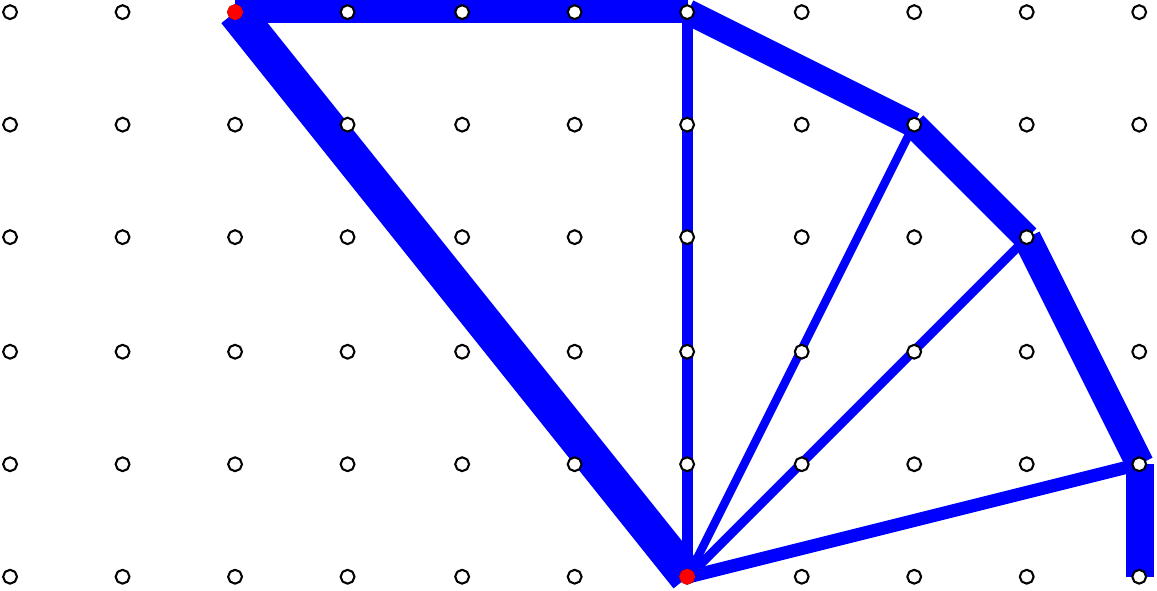}
  }
  \par
  \subfloat[]{
  \label{fig:x10_y5_gap2_downward_no_slen_deg5}
  \includegraphics[scale=0.45]{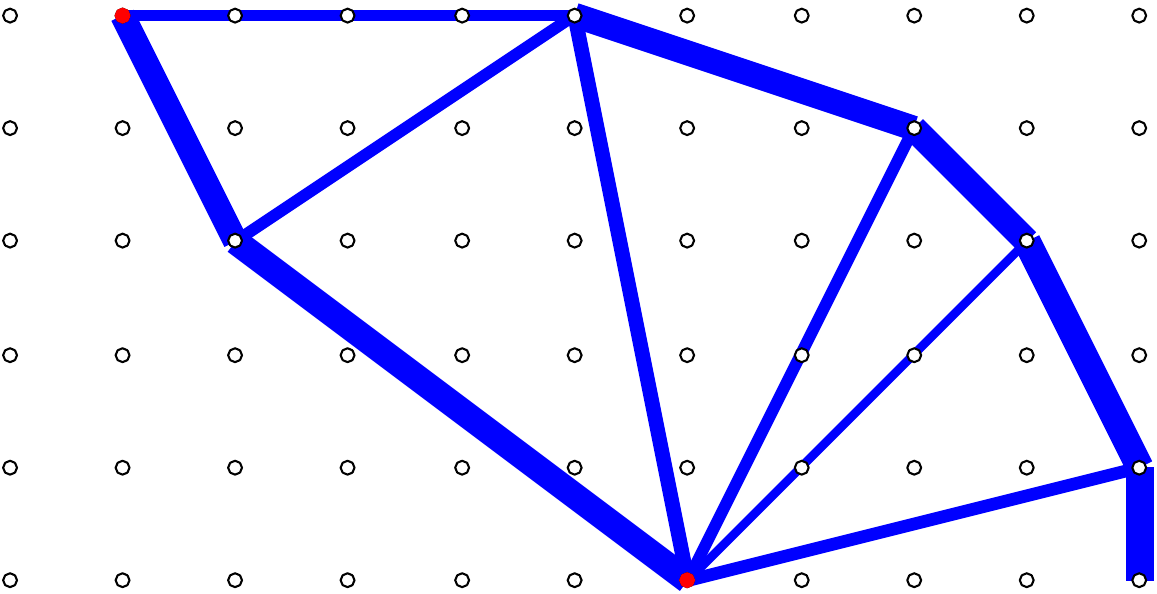}
  }
  \quad
  \subfloat[]{
  \label{fig:x10_y5_gap3_downward_no_slen_deg5}
  \includegraphics[scale=0.45]{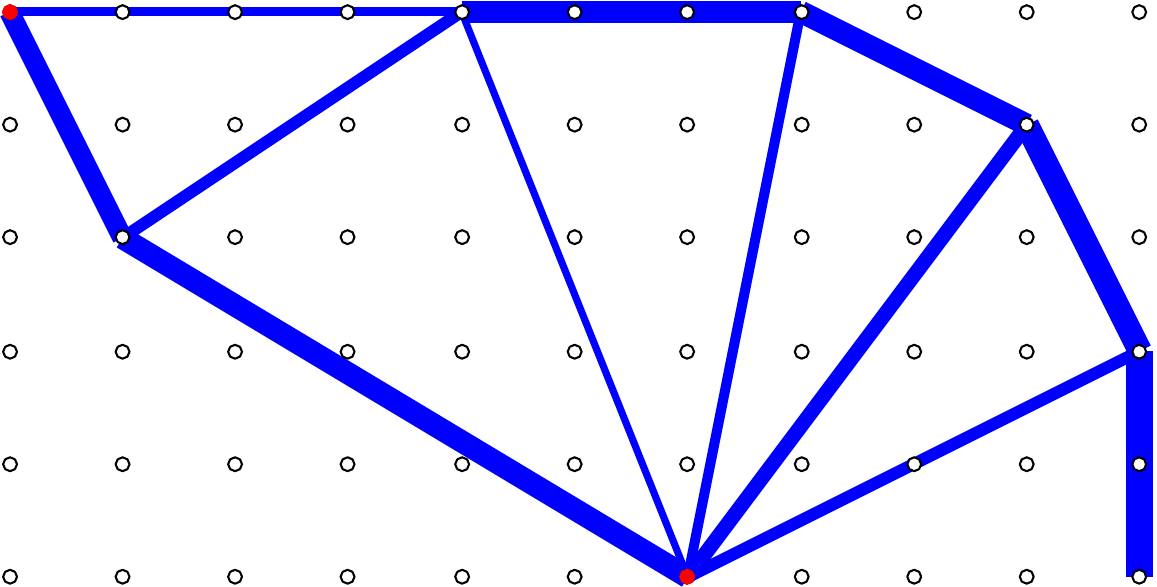}
  }
  \caption[]{The optimal truss designs obtained with 
  $\delta^{\rr{max}}=5$ and
  $(x^{\rr{min}},x^{\rr{max}})=(500,5000)\,\mathrm{mm^{2}}$ (example (I)). 
  The initial gaps are 
  \subref{fig:x10_y5_gap0_downward_no_slen_deg5}~$g_{j}=0$; 
  \subref{fig:x10_y5_gap1_downward_no_slen_deg5}~$g_{j}=0.25\,\mathrm{mm}$; 
  \subref{fig:x10_y5_gap2_downward_no_slen_deg5}~$g_{j}=0.5\,\mathrm{mm}$; and 
  \subref{fig:x10_y5_gap3_downward_no_slen_deg5}~$g_{j}=0.75\,\mathrm{mm}$. }
  \label{fig:x10_y5_gap*_downward_no_slen_deg5}
\bigskip
  \centering
  \subfloat[]{
  \label{fig:x10_y5_gap0_downward_no_slen_deg4}
  \includegraphics[scale=0.45]{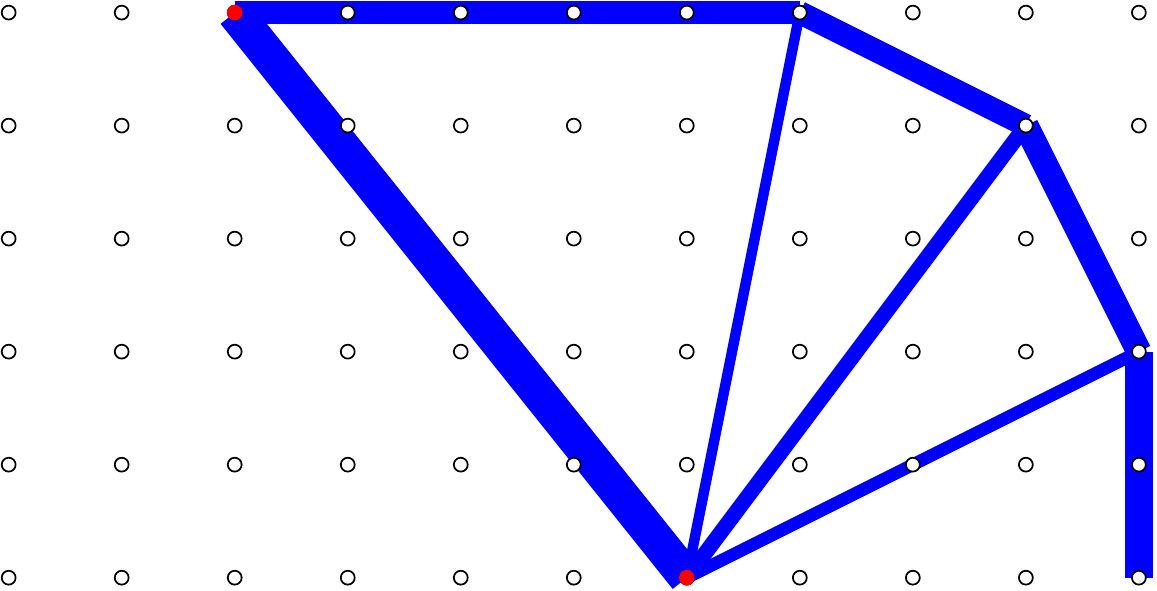}
  }
  \quad
  \subfloat[]{
  \label{fig:x10_y5_gap1_downward_no_slen_deg4}
  \includegraphics[scale=0.45]{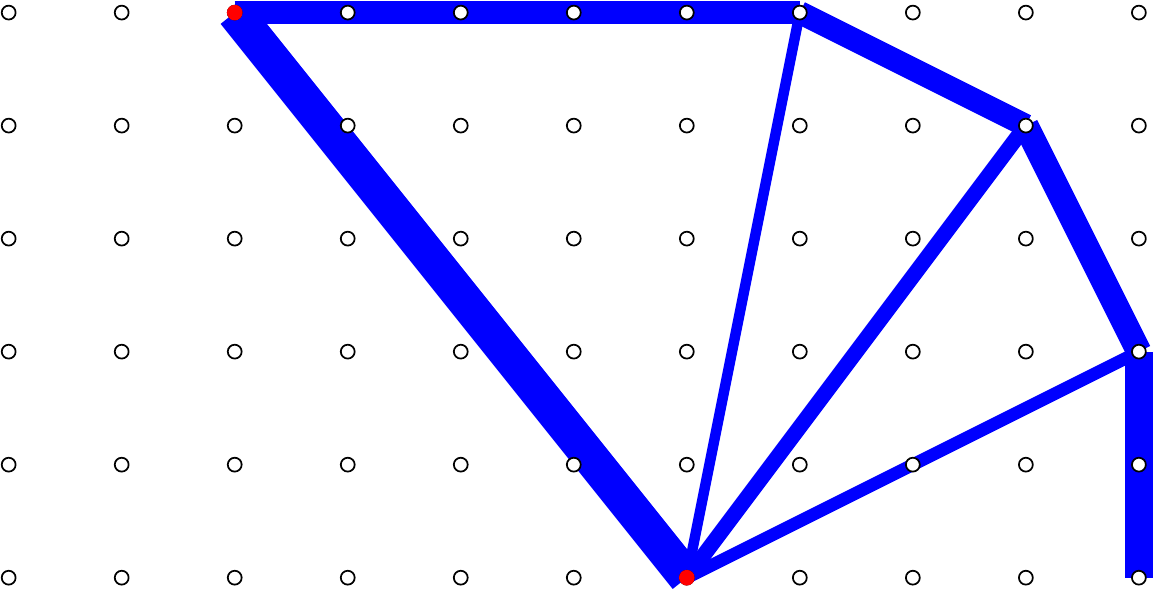}
  }
  \par
  \subfloat[]{
  \label{fig:x10_y5_gap2_downward_no_slen_deg4}
  \includegraphics[scale=0.45]{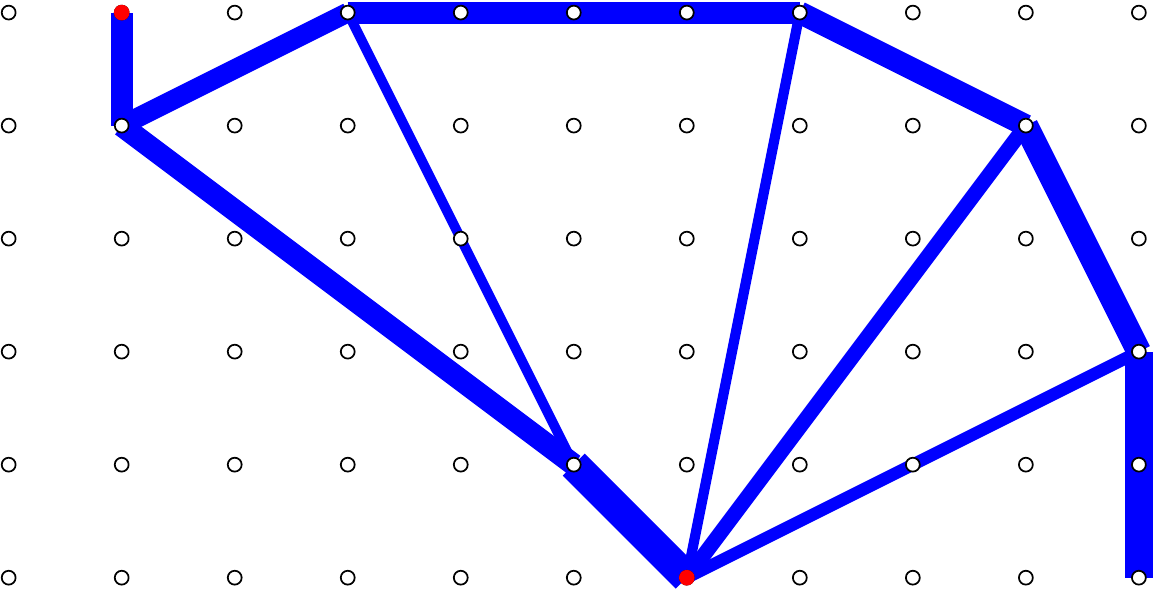}
  }
  \quad
  \subfloat[]{
  \label{fig:x10_y5_gap3_downward_no_slen_deg4}
  \includegraphics[scale=0.45]{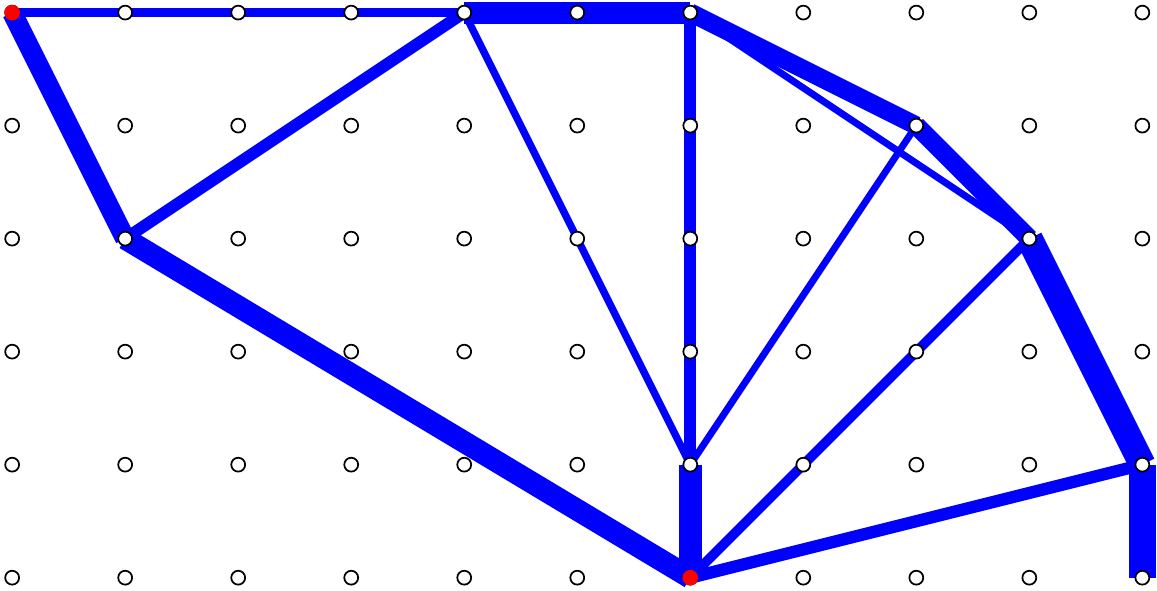}
  }
  \caption[]{The optimal truss designs obtained with 
  $\delta^{\rr{max}}=4$ and
  $(x^{\rr{min}},x^{\rr{max}})=(500,5000)\,\mathrm{mm^{2}}$ (example (I)). 
  The initial gaps are 
  \subref{fig:x10_y5_gap0_downward_no_slen_deg4}~$g_{j}=0$; 
  \subref{fig:x10_y5_gap1_downward_no_slen_deg4}~$g_{j}=0.25\,\mathrm{mm}$; 
  \subref{fig:x10_y5_gap2_downward_no_slen_deg4}~$g_{j}=0.5\,\mathrm{mm}$; and 
  \subref{fig:x10_y5_gap3_downward_no_slen_deg4}~$g_{j}=0.75\,\mathrm{mm}$. }
  \label{fig:x10_y5_gap*_downward_no_slen_deg4}
\bigskip
  \centering
  \includegraphics[scale=0.45]{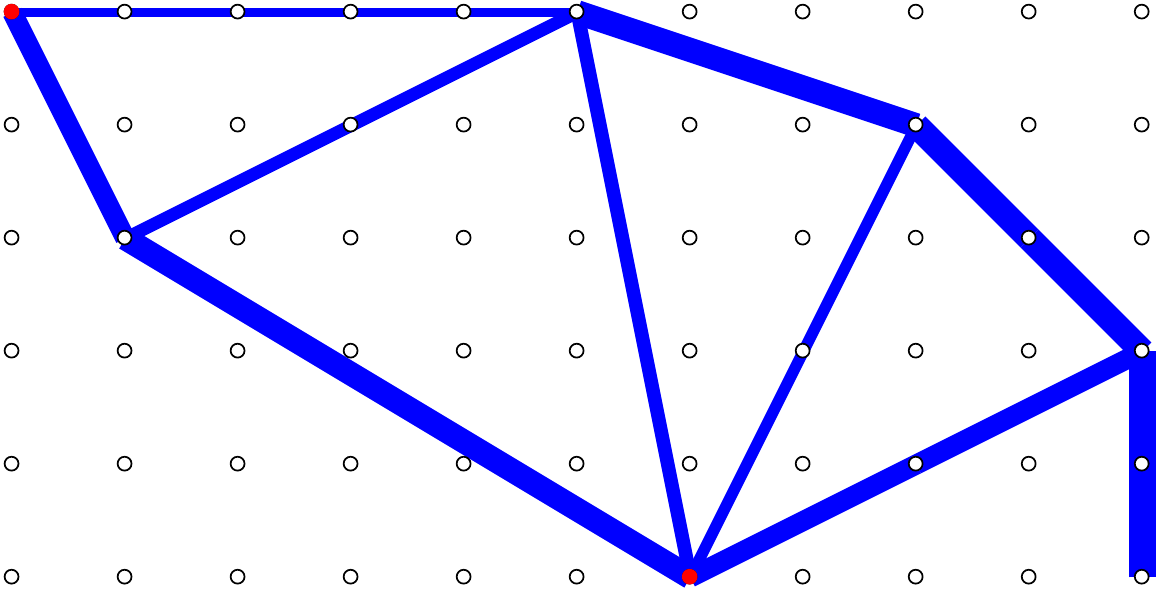}
  \caption{The optimal solution for the problem setting 
  \reffig{fig:x10_y5_gap3_downward_no_slen_deg4} with the constraints 
  prohibiting presence of mutually crossing members (example (I)). }
  \label{fig:x10_y5_no_cross_gap3_downward_no_slen_deg4}
\end{figure}

\begin{table}[bp]
  \centering
  \caption{Computational results of example (I). }
  \label{tab:ex.result.I}
  \begin{tabular}{rrrrrrrr}
    \toprule
    Contact & $g_{j}$ (mm) & $x^{\rr{min}}$ ($\mathrm{mm^{2}}$) 
    & $\delta^{\rr{max}}$ & Crssng.\ 
    & Obj.\ (J) & Time (s) & {\#}BnB-node \\
    \cmidrule(r){1-5}
    \cmidrule(l){6-8}
    Bi-latl.\  
    & $0.00$ & $0$ & $+\infty$ & acpt.\ 
    & $579.204506$ & $0.2$ & --- \\
    Uni-latl.\ 
    & $0.00$ & $0$ & $+\infty$ & acpt.\ 
    & $597.530875$ & $0.3$ & --- \\
    Uni-latl.\ 
    & $0.25$ & $0$ & $+\infty$ & acpt.\ 
    & $747.530865$ & $0.3$ & --- \\
    Uni-latl.\ 
    & $0.50$ & $0$ & $+\infty$ & acpt.\ 
    & $878.623871$ & $0.3$ & --- \\
    Uni-latl.\ 
    & $0.75$ & $0$ & $+\infty$ & acpt.\ 
    & $996.289422$ & $0.3$ & --- \\
    Uni-latl.\ 
    & $0.00$ & $500$ & $5$ & acpt.\ 
    & $600.794536$ & $43.5$ & $89$ \\
    Uni-latl.\ 
    & $0.25$ & $500$ & $5$ & acpt.\ 
    & $750.794549$ & $555.8$ & $1279$ \\
    Uni-latl.\ 
    & $0.50$ & $500$ & $5$ & acpt.\ 
    & $889.218107$ & $3909.5$ & $9838$ \\
    Uni-latl.\ 
    & $0.75$ & $500$ & $5$ & acpt.\ 
    & $1006.863373$ & $1375.9$ & $3434$ \\
    Uni-latl.\ 
    & $0.00$ & $500$ & $4$ & acpt.\ 
    & $606.950419$ & $1501.7$ & $3902$ \\
    Uni-latl.\ 
    & $0.25$ & $500$ & $4$ & acpt.\ 
    & $756.950422$ & $2208.3$ & $5809$ \\
    Uni-latl.\ 
    & $0.50$ & $500$ & $4$ & acpt.\ 
    & $896.874060$ & $27576.6$ & $83198$ \\
    Uni-latl.\ 
    & $0.75$ & $500$ & $4$ & acpt.\ 
    & $1015.913424$ & $12992.3$ & $39821$ \\
    Uni-latl.\ 
    & $0.75$ & $500$ & $4$ & proh.\ 
    & $1023.750780$ & $6739.3$ & $25258$ \\
    \bottomrule
  \end{tabular}
\end{table}

In this section, we solve the truss topology optimization problem in 
\eqref{P.truss.SOCP.1}, with incorporating some additional design 
constraints. 
Associated with each truss member, we introduce a binary variable to 
indicate whether the member exists or vanishes. 
Then, we can consider various design constraints, including 
lower bound constraints on the cross-sectional areas of existing 
members \citep{KY17}, 
an upper bound constraint on the number of nodes \citep{KF18}, 
limitation of the number of different values of cross-sectional 
areas \citep{Kan16,Kan19admm}, 
and upper bound constraints on the degrees of nodes \citep{KOG19}. 
All of these constraints can be treated within the framework of 
{\em mixed-integer second-order cone programming\/} (MISOCP). 

SOCP and MISOCP problems were solved with CPLEX ver.~12.8.0 \citep{CPLEX}, 
where the problem data were prepared in the CPLEX LP file format by 
the code implemented in Matlab ver.~9.0.0.  

As for MIP parameters of CPLEX, we set 
\texttt{mip tolerances integrality} (the amount by which each 0-1 
variable can be different from an integer) to $10^{-6}$, 
\texttt{mip tolerances mipgap} (the relative tolerance on the gap 
between the best feasible objective value and the objective value of the 
best node remaining) to $10^{-6}$, 
\texttt{emphasis numerical} (the emphasis on numerical precision) 
to \texttt{yes}, 
\texttt{emphasis mip} to $4$ (i.e., finding hidden feasible solutions is 
emphasized). 

\reffig{fig:initial_truss} depicts a problem setting. 
The nodes of the truss are aligned on 
a $1\,\mathrm{m} \times 1\,\mathrm{m}$ grid. 
Any pair of two nodes is connected by members aligned in a straight manner, 
where overlapping of members is avoided by removing the longer member 
when two members overlap. 
Accordingly, the truss has $m=1361$ members and $n=132$ degrees of 
freedom of the nodal displacements. 
It has no node at which the nodal displacement is prescribed, i.e., $d=0$. 
The number of contact candidate nodes is $c=20$. 
The initial gaps between the contact candidate nodes and the obstacle 
surface are uniformly distributed, i.e., $g_{1}=\dots=g_{c}$. 
The Young modulus of the members is $E=20\,\mathrm{GPa}$. 
The downward vertical external force of $100\,\mathrm{kN}$ is 
applied at the right bottom node. 

\reffig{fig:x10_y5_gap*_downward_cont} collects the obtained optimal 
solutions of problem \eqref{P.truss.SOCP.1} with various values of 
initial gaps. 
Here, the width of each member is proportional to its 
cross-sectional area, and filled circles indicate the contact candidate 
nodes that are in contact at the equilibrium state; the remaining 
contact candidate nodes are free. 
In every solution, only two nodes are in contact. 
For comparison, if we suppose that the contact conditions are bilateral 
(i.e., all the contact candidate nodes are replaced by roller supports), 
then the solution shown in \reffig{fig:x10_y5_gap0_downward_bi_cont} is 
obtained. 
The computational results of these examples are listed in the first five 
rows of \reftab{tab:ex.result.I}. 
Here, ``obj.''\ reports the optimal value, and 
``time'' is the computational time spent by CPLEX. 

We next consider an upper bound constraint on the degree of 
each node \citep{KOG19}.\footnote{%
The degree of a node is the number of members connected to the node. } 
It is worth noting that, in every solution shown in 
\reffig{fig:x10_y5_gap*_downward_cont}, the degree of the bottom center 
node is $7$. 
We use $\delta^{\rr{max}}$ to denote the specified upper bound for the 
nodal degree. 
Also, we consider  the lower and upper bounds, denoted $x^{\rr{min}}$ 
and $x^{\rr{max}}$, respectively, for the member cross-sectional areas. 
Namely, we incorporate the constraints \citep{KY17,Kan16} 
\begin{align*}
  x_{e} \in \{ 0 \} \cup [ x^{\rr{min}}, x^{\rr{max}} ] , 
  \quad e=1,\dots,m
\end{align*}
with $x^{\rr{min}}=500\,\mathrm{mm^{2}}$ 
and $x^{\rr{max}}=5000\,\mathrm{mm^{2}}$. 
\reffig{fig:x10_y5_gap*_downward_no_slen_deg5} and 
\reffig{fig:x10_y5_gap*_downward_no_slen_deg4} collects the obtained 
optimal solutions with $\delta^{\rr{max}}=5$ and  $\delta^{\rr{max}}=4$, 
respectively. 
The computational results are reported in \reftab{tab:ex.result.I}, where 
``{\#}BnB-node'' is the number of enumeration nodes explored by CPLEX. 

We focus attention on the solution in 
\reffig{fig:x10_y5_gap3_downward_no_slen_deg4}, which has a pair of 
mutually crossing members. 
By incorporating the constraints that prohibit presence of mutually 
crossing members \citep{Kan16}, we obtain the optimal solution shown in 
\reffig{fig:x10_y5_no_cross_gap3_downward_no_slen_deg4}. 
The computational results are reported in the bottom row of 
\reftab{tab:ex.result.I}. 

Thus, we can solve optimization problems with various design constraints, 
because our formulation in \eqref{P.truss.SOCP.1} is convex and retains 
the member cross-sectional areas as explicit optimization variables.

\subsection{Example (II): Continua}
\label{sec:ex.2}

\begin{figure}[tbp]
  \centering
  \includegraphics[scale=0.40]{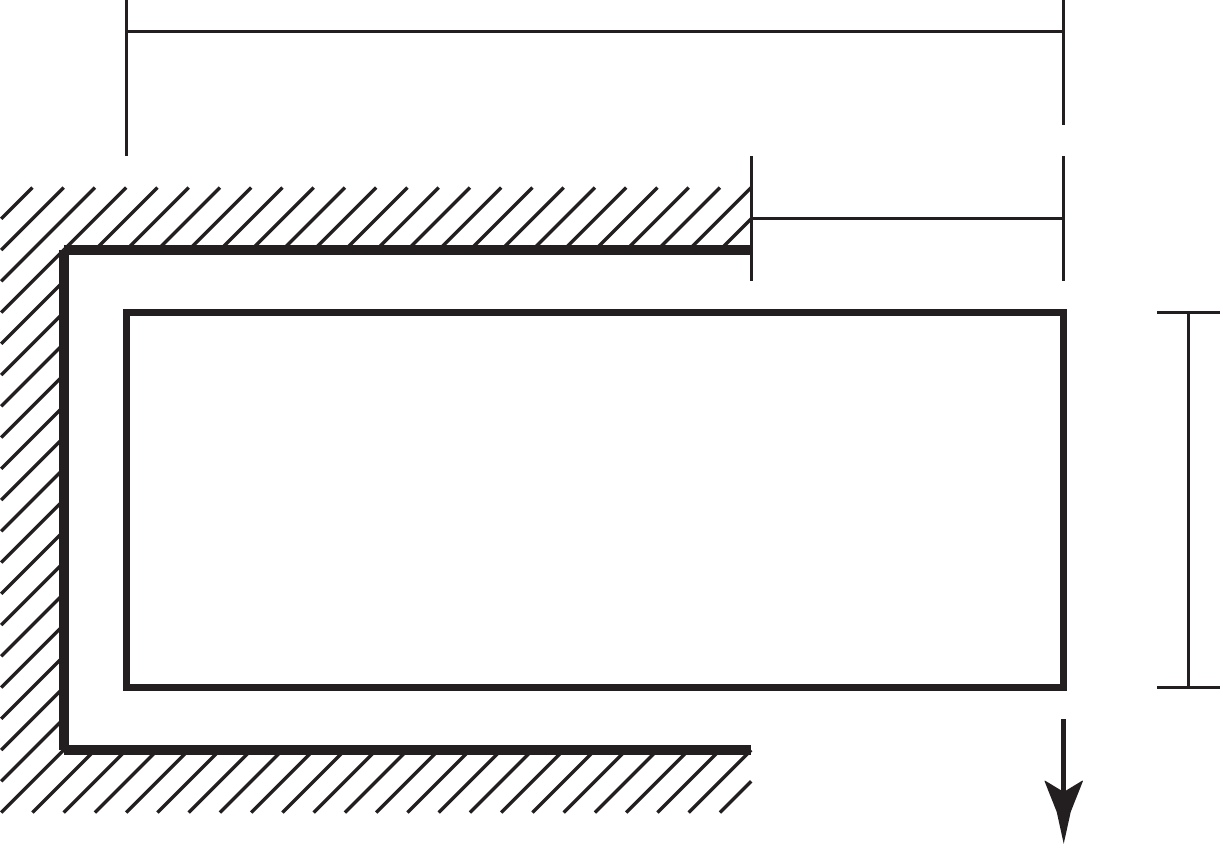}
  \begin{picture}(0,0)
    \put(-150,-50){
    \put(67,136.5){{\footnotesize $80$ }}
    \put(104,126.5){{\footnotesize $27$ }}
    \put(145,88){{\footnotesize $32$ }}
    }
  \end{picture}
  \caption{Problem setting of example (II).}
  \label{fig:ex_body_gap}
\end{figure}
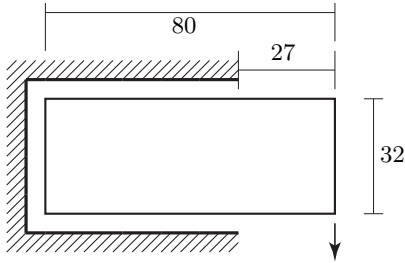

\begin{figure}[tbp]
  \centering
  \subfloat[]{
  \label{fig:x80_y32_gap0_downward}
  \includegraphics[scale=0.50]{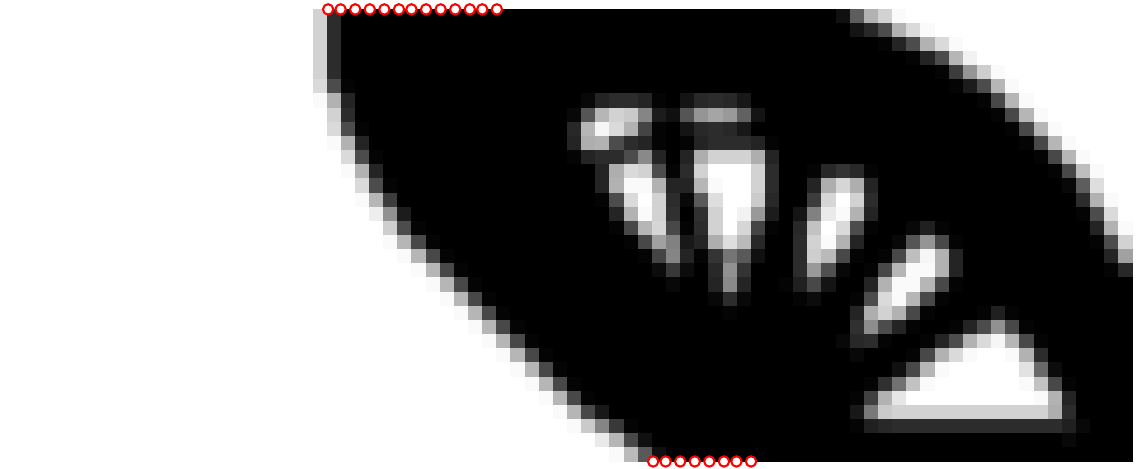}
  }
  \quad
  \subfloat[]{
  \label{fig:x80_y32_gap2_downward}
  \includegraphics[scale=0.50]{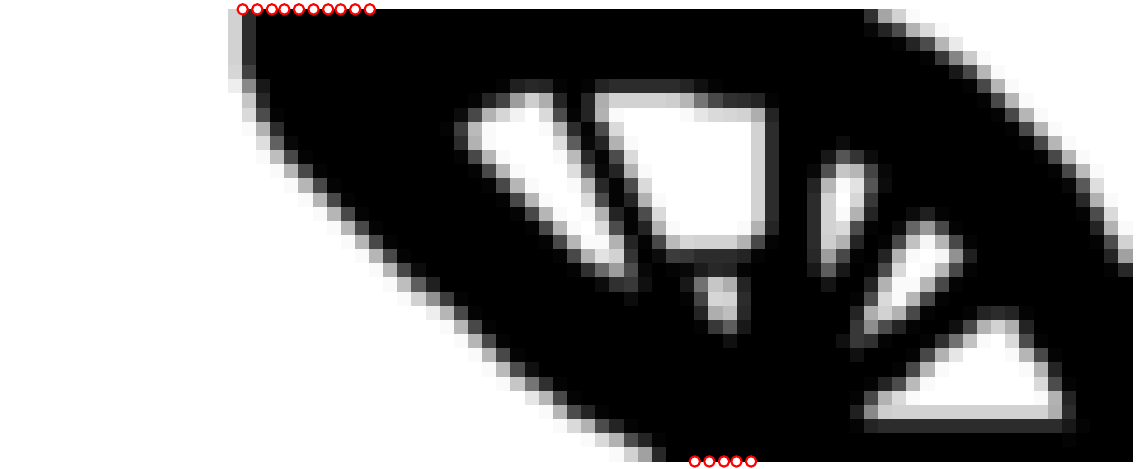}
  }
  \par
  \subfloat[]{
  \label{fig:x80_y32_gap4_downward}
  \includegraphics[scale=0.50]{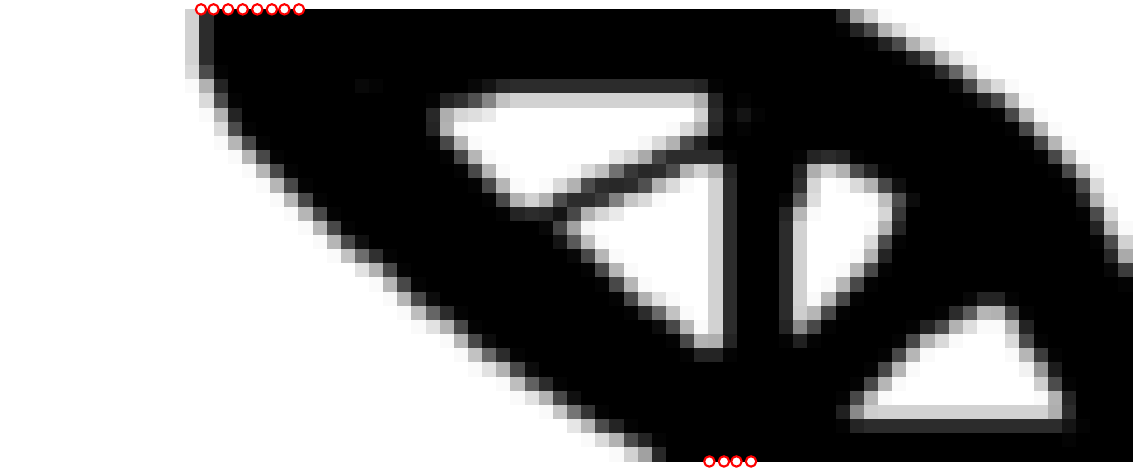}
  }
  \quad
  \subfloat[]{
  \label{fig:x80_y32_gap6_downward}
  \includegraphics[scale=0.50]{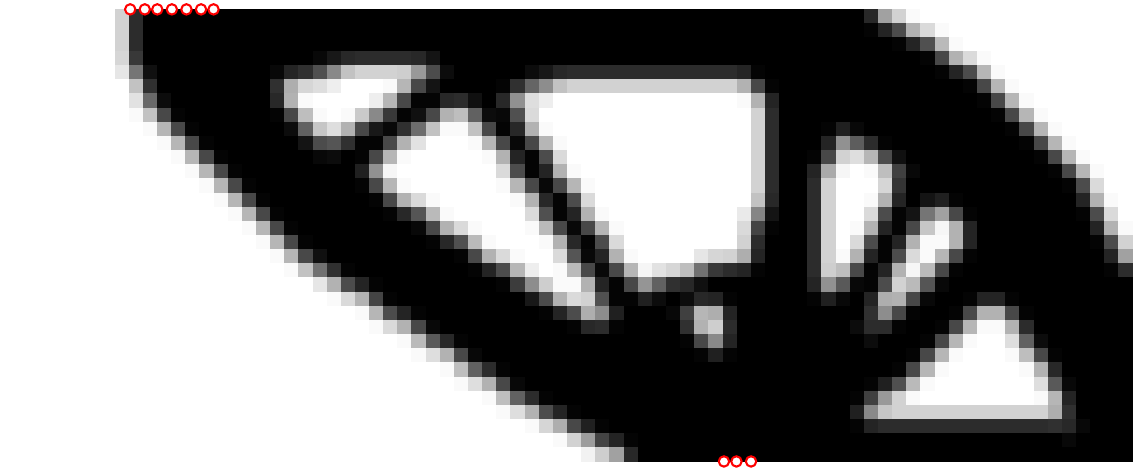}
  }
  \par
  \subfloat[]{
  \label{fig:x80_y32_gap8_downward}
  \includegraphics[scale=0.50]{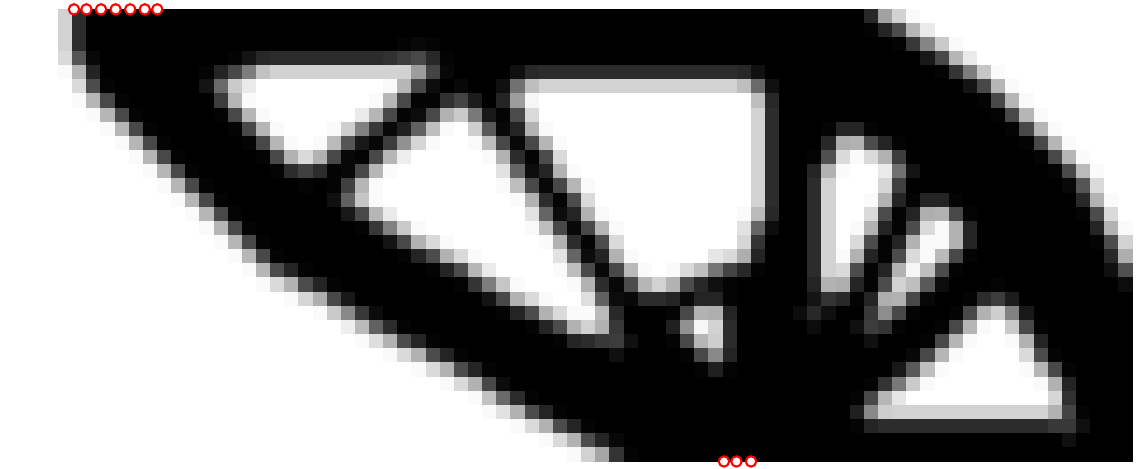}
  }
  \quad
  \subfloat[]{
  \label{fig:x80_y32_gap10_downward}
  \includegraphics[scale=0.50]{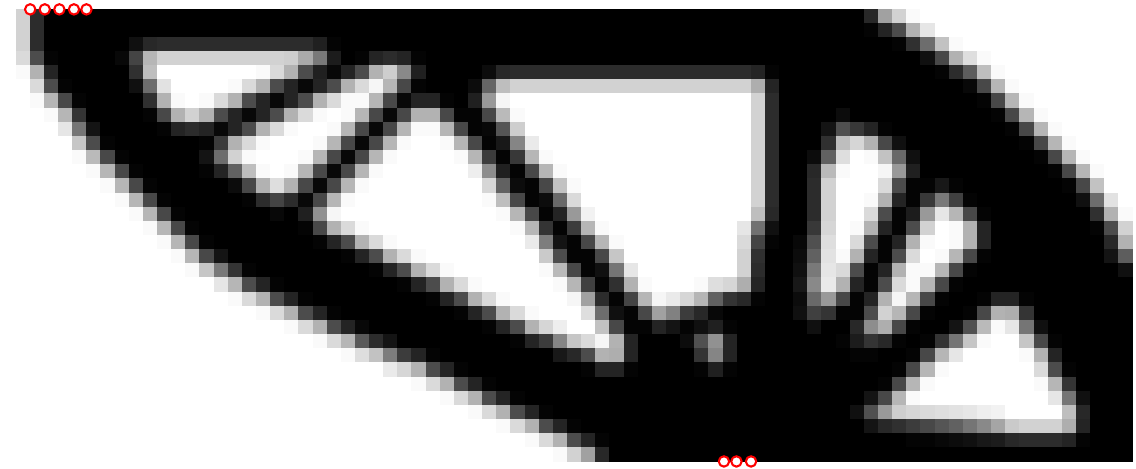}
  }
  \caption[]{%
  The solutions obtained for example (II) with the downward external load. 
  The initial gaps are 
  \subref{fig:x80_y32_gap0_downward}~$g_{j}=0$; 
  \subref{fig:x80_y32_gap2_downward}~$g_{j}=2$; 
  \subref{fig:x80_y32_gap4_downward}~$g_{j}=4$; 
  \subref{fig:x80_y32_gap6_downward}~$g_{j}=6$; 
  \subref{fig:x80_y32_gap8_downward}~$g_{j}=8$; and 
  \subref{fig:x80_y32_gap10_downward}~$g_{j}=10$. 
  }
  \label{fig:x80_y32_gap*_downward}
\end{figure}

\begin{figure}[tbp]
  \centering
  \subfloat[]{
  \label{fig:x80_y32_gap0_upward}
  \includegraphics[scale=0.50]{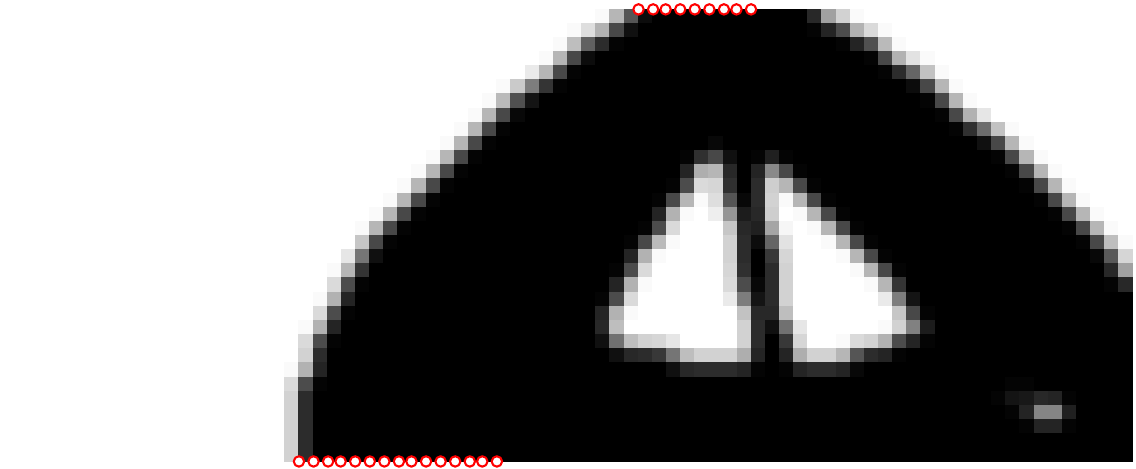}
  }
  \quad
  \subfloat[]{
  \label{fig:x80_y32_gap2_upward}
  \includegraphics[scale=0.50]{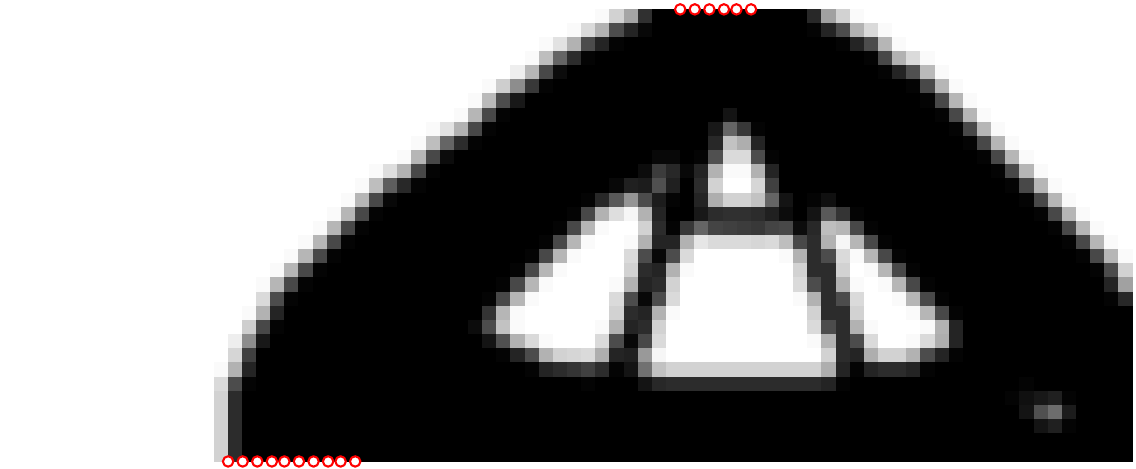}
  }
  \par
  \subfloat[]{
  \label{fig:x80_y32_gap4_upward}
  \includegraphics[scale=0.50]{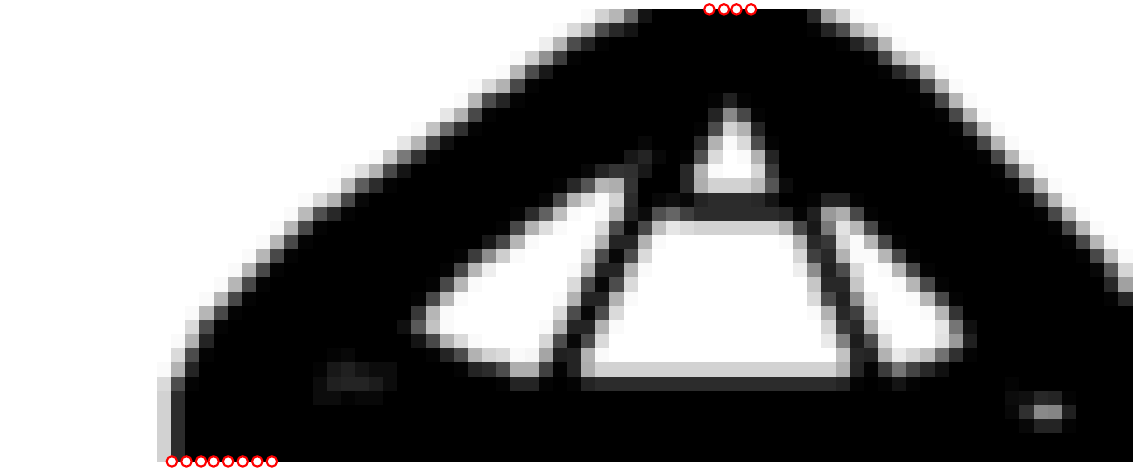}
  }
  \quad
  \subfloat[]{
  \label{fig:x80_y32_gap6_upward}
  \includegraphics[scale=0.50]{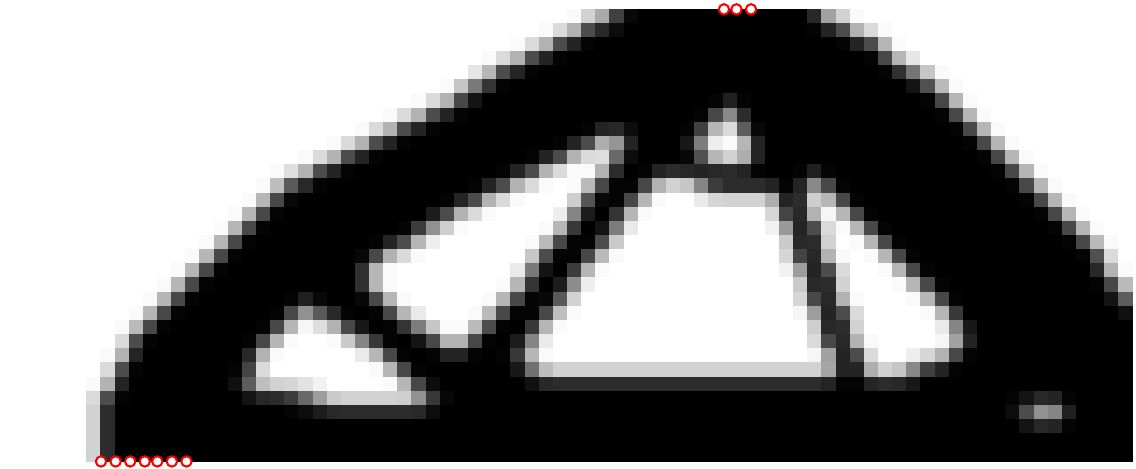}
  }
  \caption[]{%
  The solutions obtained for example (II) with the downward external load. 
  The initial gaps are 
  \subref{fig:x80_y32_gap0_upward}~$g_{j}=0$; 
  \subref{fig:x80_y32_gap2_upward}~$g_{j}=2$; 
  \subref{fig:x80_y32_gap4_upward}~$g_{j}=4$; and 
  \subref{fig:x80_y32_gap6_upward}~$g_{j}=6$. 
  }
  \label{fig:x80_y32_gap*_upward}
\end{figure}

\begin{table}[bp]
  \centering
  \caption{Computational results of example (II). }
  \label{tab:ex.result.II}
  \begin{tabular}{lrrrr}
    \toprule
    Ext.\ load & $g_{j}$ & Obj.\ & Time (s) & {\#}Iter.\ \\
    \midrule
    Downward & $0$ 
    & $29.2195$ & $1417.1$ & $51$ \\
    Downward & $2$ 
    & $41.2819$ & $2278.9$ & $81$ \\
    Downward & $4$ 
    & $51.7520$ & $2624.4$ & $79$ \\
    Downward & $6$ 
    & $61.2012$ & $2037.2$ & $61$ \\
    Downward & $8$ 
    & $69.9158$ & $3264.0$ & $115$ \\
    Downward & $10$ 
    & $78.2713$ & $3043.1$ & $95$ \\
    \midrule
    Upward & $0$ 
    & $27.4079$ & $2858.1$ & $113$ \\
    Upward & $2$ 
    & $39.4283$ & $1387.2$ & $44$ \\
    Upward & $4$ 
    & $49.5674$ & $1173.8$ & $34$ \\
    Upward & $6$ 
    & $58.6082$ & $2853.4$ & $103$ \\
    \bottomrule
  \end{tabular}
\end{table}

Consider a problem instance outlined in \reffig{fig:ex_body_gap}. 
The rectangular elastic body is discretized as $80 \times 32$ mesh, 
i.e., $m=2560$. 
There exists no node at which the nodal displacement is prescribed, i.e., 
$d=0$ and $n=5346$. 
The number of contact candidate nodes is $c=147$. 
The initial gaps between the elastic body and the obstacle surface are 
uniformly distributed, i.e., $g_{1}=\dots=g_{c}$. 
We omit the units of quantities for simplicity, as often done in 
literature on continuum-based topology optimization. 
The Young modulus and the Poisson ratio are $1$ and $0.3$, respectively. 
The external force of $1$ is applied at the right bottom corner. 

The sequential SOCP described in section~\ref{sec:body} was implemented 
in Matlab ver.~9.0.0. 
The implementation of the SIMP approach was based on the 
Matlab 88 line code in \cite{ACSLS11}. 
The penalization power and the filter radius divided by the element 
size are $p=3$ and $1.5$, respectively. 
The specified upper bound for the volume fraction is $0.5$. 
As for the initial point for the sequential SOCP, 
$\rho_{1}^{(0)},\dots,\rho_{m}^{(0)}$ are set to the volume fraction. 
We solved SOCP subproblems with SeDuMi ver.~1.3 \citep{Stu99,Pol05}, 
which implements a primal-dual interior-point method. 

\reffig{fig:x80_y32_gap*_downward} collects the obtained solutions for 
various values of the initial gaps, when the downward external force is 
applied. 
\reffig{fig:x80_y32_gap*_upward} collects the solutions obtained for the 
upward external force. 
\reftab{tab:ex.result.II} reports the computational results. 
Here, ``obj.''\ is the obtained objective value, 
``time'' is the computational time, and 
``{\#}iter.''\ is the number of iterations of the sequential SOCP (i.e., 
the number of SOCP subproblems solved before convergence). 

In \reffig{fig:x80_y32_gap*_downward} and 
\reffig{fig:x80_y32_gap*_upward}, 
small open circles indicate the contact candidate nodes that are in 
contact at the equilibrium state; the remaining contact candidate nodes 
are free. 
We can observe that the larger the initial gaps, the wider the optimal 
configuration becomes.

\section{Concluding remarks}
\label{sec:conclusions}

This paper has presented new formulations for topology optimization of 
structures with frictionless unilateral contact. 
Specifically, for trusses, we have seen that the stiffness maximization 
problem can be recast as an SOCP (second-order cone programming) problem. 
Also, for continua, the stiffness maximization problem is formulated as 
a nonlinear SOCP problem. 
This nonlinear SOCP formulation is suitable for application of a simple 
successive linearization method, in which we solve an SOCP subproblem at 
each iteration. 
The key to deriving these formulations is the Lagrange duality in the 
optimization problem defining the compliance of a structural design. 

Unlike the convex optimization approach in \citet{KZZ98} for trusses, 
the SOCP formulation presented in this paper retains the 
member cross-sectional areas as optimization variables. 
Therefore, various additional design constraints, e.g., limitation of the 
number of different cross-sectional areas \citep{Kan16} and limitation 
of the number of nodes \citep{KF18}, 
as well as the self-weight load \citep{KY17}, can be incorporated into 
this SOCP formulation. 
Problems with such constraints are formulated as MISOCP 
(mixed-integer second-order cone programming) problems, which can be 
solved globally with a standard mixed-integer programming solver. 
Extension of the presented approach to frame topology optimization with 
discrete design variables \citep{Kan16frame} is straightforward.

\paragraph{Acknowledgments}

The work described in this paper is partially supported by 
JSPS KAKENHI 17K06633.


\begin{thebibliography}{99}
\bibitem[\protect\citeauthoryear{Andreassen~{\em et al.\/}}{2011}]{ACSLS11}
  {E.~Andreassen, A.~Clausen, M.~Schevenels, B.~S.~Lazarov, O.~Sigmund}:
  {Efficient topology optimization in MATLAB using 88 lines of code}.
  {\em Structural and Multidisciplinary Optimization},
  \textbf{43}, 1--16 (2011).

\bibitem[\protect\citeauthoryear{Anjos and Lasserre}{2012}]{AL12}
  {M.~F.~Anjos, J.~B.~Lasserre (eds.)}:
  {\em Handbook on Semidefinite, Conic and Polynomial Optimization}.
  Springer, New York (2012).

\bibitem[\protect\citeauthoryear{Ben-Tal~{\em et al.\/}}{2000}]{BKNZ00}
  {A.~Ben-Tal, M.~Ko\v{c}vara, A.~Nemirovski, J.~Zowe}:
  {Free material design via semidefinite programming: 
    the multiload case with contact conditions}.
  {\em SIAM Review},
  \textbf{42}, 695--715 (2000).

\bibitem[\protect\citeauthoryear{Bends{\o}e and Sigmund}{1999}]{BS99}
  {M.~P.~Bends{\o}e, O.~Sigmund}:
  {Material interpolation schemes in topology optimization}.
  {\em Archive of Applied Mechanics},
  \textbf{69}, 635--654 (1999).

\bibitem[\protect\citeauthoryear{Bourdin}{2001}]{Bou01}
  {B.~Bourdin}:
  {Filters in topology optimization}.
  {\em International Journal for Numerical Methods in Engineering},
  \textbf{50}, 2143--2158 (2001).


\bibitem[\protect\citeauthoryear{Bruns and Tortorelli}{2001}]{BT01}
  {T.~E.~Bruns, D.~A.~Tortorelli}:
  {Topology optimization of non-linear elastic structures 
    and compliant mechanisms}.
  {\em Computer Methods in Applied Mechanics and Engineering},
  \textbf{190}, 3443--3459 (2001).

\bibitem[\protect\citeauthoryear{Ciarlet}{1989}]{Cia89}
  {P.~G.~Ciarlet}:
  {\em Introduction to Numerical Linear Algebra and Optimisation}.
  Cambridge University Press, Cambridge (1989).

\bibitem[\protect\citeauthoryear{Ekeland and T\'{e}mam}{1976}]{ET76}
  {I.~Ekeland, R.~T\'{e}mam}:
  {\em Convex Analysis and Variational Problems}.
  North-Holland, Amsterdam (1976);
  SIAM, Philadelphia (1999).

\bibitem[\protect\citeauthoryear{Fancello}{2006}]{Fan06}
  {E.~A.~Fancello}:
  {Topology optimization for minimum mass design 
    considering local failure constraints and contact boundary conditions}.
  {\em Structural and Multidisciplinary Optimization},
  \textbf{32}, 229--240 (2006).

\bibitem[\protect\citeauthoryear{Geniaut~{\em et al.\/}}{2007}]{GMM07}
  {S.~Geniaut, P.~Massin, N.~M\"{o}es}:
  {A stable 3D contact formulation using X-FEM}.
  {\em European Journal of Computational Mechanics},
  \textbf{16}, 259--275 (2007).


\bibitem[\protect\citeauthoryear{Hilding}{2000}]{Hil00}
  {D.~Hilding}:
  {A heuristic smoothing procedure for avoiding local optima 
    in optimization of structures subject to unilateral constraints}.
  {\em Structural and Multidisciplinary Optimization},
  \textbf{20}, 29--36 (2000).

\bibitem[\protect\citeauthoryear{Hilding and Klarbring}{2012}]{HK12}
  {D.~Hilding, A.~Klarbring}:
  {Optimization of structures in frictional contact}.
  {\em Computer Methods in Applied Mechanics and Engineering},
  \textbf{205--208}, 83--90 (2012).

\bibitem[\protect\citeauthoryear{Hilding~{\em et al.\/}}{1999}]{HKP99}
  {D.~Hilding, A.~Klarbring, J.-S.~Pang}:
  {Minimization of maximum unilateral force}.
  {\em Computer Methods in Applied Mechanics and Engineering},
  \textbf{177}, 215--234 (1999).

\bibitem[\protect\citeauthoryear{Hilding~{\em et al.\/}}{1999}]{HKP99survey}
  {D.~Hilding, A.~Klarbring, J.~Petersson}:
  {Optimization of structures in unilateral contact}.
  {\em Applied Mechanics Reviews},
  \textbf{52}, 139--160 (1999).

\bibitem[\protect\citeauthoryear{IBM ILOG}{2019}]{CPLEX}
  {IBM ILOG}:
  {\em IBM ILOG CPLEX Optimization Studio Documentation}.
  \url{http://www.ibm.com/support/knowledgecenter/}
  (Accessed January 2019).

\bibitem[\protect\citeauthoryear{Kanno}{2011}]{Kan11}
  {Y.~Kanno}:
  {\em Nonsmooth Mechanics and Convex Optimization\/}.
  CRC Press, Boca Raton (2011). 

\bibitem[\protect\citeauthoryear{Kanno}{2016a}]{Kan16}
  {Y.~Kanno}:
  {Global optimization of trusses 
    with constraints on number of different cross-sections: 
    a mixed-integer second-order cone programming approach}.
  {\em Computational Optimization and Applications},
  \textbf{63}, 203--236 (2016).

\bibitem[\protect\citeauthoryear{Kanno}{2016b}]{Kan16frame}
  {Y.~Kanno}:
  {Mixed-integer second-order cone programming for global 
    optimization of compliance of frame structure 
    with discrete design variables}.
  {\em Structural and Multidisciplinary Optimization},
  \textbf{54}, 301--316 (2016).

\bibitem[\protect\citeauthoryear{Kanno}{2019}]{Kan19admm}
  {Y.~Kanno}:
  {Alternating direction method of multipliers as simple heuristic 
    for topology optimization of a truss with uniformed member cross-sections}.
  {\em Journal of Mechanical Design (ASME)},
  \textbf{141}, Article No.~011403 (2019).


\bibitem[\protect\citeauthoryear{Kanno and Fujita}{2018}]{KF18}
  {Y.~Kanno, S.~Fujita}:
  {Alternating direction method of multipliers 
    for truss topology optimization with limited number of nodes: 
    a cardinality-constrained second-order cone programming approach}.
  {\em Optimization and Engineering},
  \textbf{19}, 327--358 (2018).

\bibitem[\protect\citeauthoryear{Kanno {\em et al.\/}}{2019}]{KOG19}
  {Y.~Kanno, M.~Ohsaki, J.~K.~Guest}:
  {Unified treatment of some different fabrication-cost functions 
    in truss topology optimization}.
  {\em Proceedings of International Association for Shell and 
    Spatial Structures (IASS) Annual Symposium 2019 and 
    Structural Membranes 2019---Form and Force},
  Barcelona, Spain, October 7--10, (2019).

\bibitem[\protect\citeauthoryear{Kanno and Takewaki}{2006}]{KT06}
  {Y.~Kanno, I.~Takewaki}:
  {Sequential semidefinite program for maximum robustness design of 
    structures under load uncertainties}.
  {\em Journal of Optimization Theory and Applications},
  \textbf{130}, 265--287 (2006).

\bibitem[\protect\citeauthoryear{Kanno and Yamada}{2017}]{KY17}
  {Y.~Kanno, H.~Yamada}:
  {A note on truss topology optimization under self-weight load:
    mixed-integer second-order cone programming approach}.
  {\em Structural and Multidisciplinary Optimization},
  \textbf{56}, 221--226 (2017).

\bibitem[\protect\citeauthoryear{Kanzow {\em et al.\/}}{2005}]{KNKF05}
  {C.~Kanzow, C.~Nagel, H.~Kato, M.~Fukushima}:
  {Successive linearization methods for nonlinear semidefinite programs}.
  {\em Computational Optimization and Applications},
  \textbf{31}, 251--273 (2005). 

\bibitem[\protect\citeauthoryear{Kim~{\em et al.\/}}{2001}]{KPC01}
  {N.~H.~Kim, Y.~H.~Park, K.~K.~Choi}:
  {Optimization of a hyperelastic structure with multibody contact 
    using continuum-based shape design sensitivity analysis}.
  {\em Structural and Multidisciplinary Optimization},
  \textbf{21}, 196--208 (2001).

\bibitem[\protect\citeauthoryear{Kim~{\em et al.\/}}{2002}]{KYC02}
  {N.~H.~Kim, K.~Yi, K.~K.~Choi}:
  {A material derivative approach in design sensitivity analysis 
    of three-dimensional contact problems}.
  {\em International Journal of Solids and Structures},
  \textbf{39}, 2087--2108 (2002).


\bibitem[\protect\citeauthoryear{Klarbring~{\em et al.\/}}{1995}]{KPR95}
  {A.~Klarbring, J.~Petersson, M.~R\"{o}nnqvist}:
  {Truss topology optimization including unilateral contact}.
  {\em Journal of Optimization Theory and Applications},
  \textbf{87}, 1--31 (1995).

\bibitem[\protect\citeauthoryear{Klarbring and R\"{o}nnqvist}{1995}]{KR95}
  {A.~Klarbring, M.~R\"{o}nnqvist}:
  {Nested approach to structural optimization in nonsmooth mechanics}.
  {\em Structural Optimization},
  \textbf{10}, 79--86 (1995).

\bibitem[\protect\citeauthoryear{Klarbring and Str\"{o}mberg}{2012}]{KS12}
  {A.~Klarbring, N.~Str\"{o}mberg}:
  {A note on the min-max formulation of stiffness optimization 
    including non-zero prescribed displacements}.
  {\em Structural and Multidisciplinary Optimization},
  \textbf{45}, 147--149 (2012).


\bibitem[\protect\citeauthoryear{Ko\v{c}vara~{\em et al.\/}}{1998}]{KZZ98}
  {M.~Ko\v{c}vara, M.~Zibulevsky, J.~Zowe}:
  {Mechanical design problems with unilateral contact}.
  {\em Mathematical Modeling and Numerical Analysis},
  \textbf{32}, 255--281 (1998).

\bibitem[\protect\citeauthoryear{Lawry and Maute}{2015}]{LM15}
  {M.~Lawry, K.~Maute}:
  {Level set topology optimization of problems with sliding contact interfaces}.
  {\em Structural and Multidisciplinary Optimization},
  \textbf{52}, 1107--1119 (2015).

\bibitem[\protect\citeauthoryear{Luo~{\em et al.\/}}{2016}]{LLK16}
  {Y.~Luo, M.~Li, Z.~Kang}:
  {Topology optimization of hyperelastic structures 
    with frictionless contact supports}.
  {\em International Journal of Solids and Structures},
  \textbf{81}, 373--382 (2016).

\bibitem[\protect\citeauthoryear{Luo~{\em et al.\/}}{1996}]{LPR96}
  {Z.-Q.~Luo, J.-S.~Pang, D.~Ralph}:
  {\em Mathematical Programs with Equilibrium Constraints}.
  Cambridge University Press, Cambridge (1996).

\bibitem[\protect\citeauthoryear{Mankame and Ananthasuresh}{2004}]{MA04}
  {N.~D.~Mankame, G.~K.~Ananthasuresh}:
  {Topology optimization for synthesis of contact-aided 
    compliant mechanisms using regularized contact modeling}.
  {\em Computers and Structures},
  \textbf{82}, 1267--1290 (2004).

\bibitem[\protect\citeauthoryear{Martins and Raous}{2002}]{MR02}
  {J.~A.~C.~Martins, M.~Raous (eds.)}:
  {\em Friction and Instabilities}.
  Springer-Verlag, Wien (2002).



\bibitem[\protect\citeauthoryear{Niu~{\em et al.\/}}{2011}]{NXC11}
  {F.~Niu, S.~Xu, G.~Cheng}:
  {A general formulation of structural topology optimization 
    for maximizing structural stiffness}.
  {\em Structural and Multidisciplinary Optimization},
  \textbf{43}, 561--572 (2011).

\bibitem[\protect\citeauthoryear{Petersson}{1996}]{Pet96}
  {J.~Petersson}:
  {On stiffness maximization of variable thickness sheet with unilateral contact}.
  {\em Quarterly of Applied Mathematics},
  \textbf{54}, 541--550 (1996).


\bibitem[\protect\citeauthoryear{Petersson and Patriksson}{1997}]{PP97}
  {J.~Petersson, M.~Patriksson}:
  {Topology optimization of sheets in contact by a subgradient method}.
  {\em International Journal for Numerical Methods in Engineering},
  \textbf{40}, 1295--1321 (1997).

\bibitem[\protect\citeauthoryear{P\'{o}lik}{2005}]{Pol05}
  {I.~P\'{o}lik}:
  {\em Addendum to the SeDuMi User Guide: Version 1.1\/}.
  Technical Report, 
  Advanced Optimization Laboratory, McMaster University, Hamilton (2005).
  \url{http://sedumi.ie.lehigh.edu/sedumi/}

\bibitem[\protect\citeauthoryear{Rockafellar}{1970}]{Roc70}
  {R.~T.~Rockafellar}:
  {\em Convex Analysis}.
  Princeton University Press, Princeton (1970).

\bibitem[\protect\citeauthoryear{Stavroulakis}{1995}]{Sta95}
  {G.~E.~Stavroulakis}:
  {Optimal prestress of cracked unilateral structures: finite element 
    analysis of an optimal control problem for variational inequalities}.
  {\em Computer Methods in Applied Mechanics and Engineering},
  \textbf{123}, 231--246 (1995).

\bibitem[\protect\citeauthoryear{Str\"{o}mberg}{2010}]{Str10}
  {N.~Str\"{o}mberg}:
  {Topology optimization of structures with manufacturing 
    and unilateral contact constraints 
    by minimizing an adjustable compliance--volume product}.
  {\em Structural and Multidisciplinary Optimization},
  \textbf{42}, 341--350 (2010).

\bibitem[\protect\citeauthoryear{Str\"{o}mberg}{2012}]{Str12}
  {N.~Str\"{o}mberg}:
  {Topology optimisation of bodies in unilateral contact 
    by maximizing the potential energy}.
  {\em Proceedings of the 11th International Conference
    on Computational Structures Technology},
  Paper No.~237, 
  Dubrovnik, Croatia (2012).

\bibitem[\protect\citeauthoryear{Str\"{o}mberg and Klarbring}{2010}]{SK10}
  {N.~Str\"{o}mberg, A.~Klarbring}:
  {Topology optimization of structures in unilateral contact}.
  {\em Structural and Multidisciplinary Optimization},
  \textbf{41}, 57--64 (2010).

\bibitem[\protect\citeauthoryear{Sturm}{1999}]{Stu99}
  {J.~F.~Sturm}:
  {Using SeDuMi 1.02, a MATLAB toolbox for optimization over symmetric cones}.
  {\em Optimization Methods and Software},
  \textbf{11--12}, 625--653 (1999).

\bibitem[\protect\citeauthoryear{Tin-Loi}{1999}]{Tin99}
  {F.~Tin-Loi}:
  {On the numerical solution of a class of unilateral contact 
    structural optimization problems}.
  {\em Structural Optimization},
  \textbf{17}, 155--161 (1999).

\bibitem[\protect\citeauthoryear{Wriggers}{2006}]{Wri06}
  {P.~Wriggers}:
  {\em Computational Contact Mechanics (2nd ed.)}.
  Springer-Verlag, Berlin (2006).

\end{thebibliography}
\end{document}